\documentclass[11pt]{amsart}
\usepackage{txfonts}
\usepackage{amsfonts}
\usepackage{amssymb}
\usepackage{amsmath}
\usepackage{mathrsfs}
\usepackage{mathtools}
\usepackage{breqn}
\usepackage[pdftex]{hyperref}

\title[Well-posedness of Weak and Strong Solutions]{Well-posedness of Weak and Strong Solutions to the Kinetic Cucker-Smale Model}

\author[C. Jin]{Chunyin Jin}
\address[Chunyin Jin]{\newline BICMR, Peking University,\newline Beijing 100871, P. R. China.}
\email{jinchunyin@pku.edu.cn}

\newtheorem{theorem}{Theorem}[section]
\newtheorem{definition}{Definition}[section]
\newtheorem{lemma}{Lemma}[section]

\newtheorem{remark}{Remark}[section]
\newcommand{\bbr}{\mathbb R}

\newcommand{\e}{\varepsilon}

\newcommand{\bx}{\mbox{\boldmath $x$}}
\newcommand{\by}{\mbox{\boldmath $y$}}

\newcommand{\bv}{\mbox{\boldmath $v$}}
\newcommand{\bw}{\mbox{\boldmath $w$}}

\newcommand{\bz}{\mbox{\boldmath $z$}}

\begin{document}

\date{\today}



\begin{abstract}
We first prove the well-posedness of weak and strong solutions to the kinetic Cucker-Smale model in the Sobolev space with the initial data having compact velocity support. Then we study the kinetic Cucker-Smale model with noise. By introducing two weighted Hilbert spaces, we successfully establish the well-posedness of weak and strong solutions, respectively. Our proof is based on weighted energy estimates. Finally, we rigorously justify the vanishing noise limit.
\end{abstract}

\maketitle \centerline{\date}
%
%
\section{Introduction}
\setcounter{equation}{0}
In this paper, we study the well-posedness of weak and strong solutions to the following kinetic Cucker-Smale model
with or without noise:
\begin{equation} \label{eq-kine}
     \begin{dcases}
         f_t + \bv \cdot \nabla_{\bx} f+ \nabla_{\bv} \cdot (L[f]f)=\sigma \Delta_{\bv}f,\\
         f|_{t=0}=f_0(\bx,\bv),
     \end{dcases}
\end{equation}
where $L[f]$ is given by
\[
 L[f](t, \bx, \bv)=\int_{\bbr^{2d}}\varphi(|\bx-\by|)f(t, \by, \bv^*)(\bv^*-\bv)d \by d \bv^*.
\]
Here $f(t,\bx,\bv)$ is the density distribution function, $\varphi(\cdot)$ is a positive non-increasing function denoting the interaction kernel, and $\sigma$ represents the noise strength. For convenience, we suppose that
\[
 \max\{|\varphi|, |\varphi'|, |\varphi''|\} \le 1 \quad  \text{and} \quad 0 \le \sigma \le 1
\]
in the sequel.

Recently, the Cucker-Smale model and related models have attracted much attention from researchers in diverse fields, including biology, physics and mathematics. People wish to understand the mechanisms that lead to the collective behaviors, such as flocking of birds, schooling of fish and swarming of bacteria, by modeling, simulation and mathematical analysis. Among them, the Cucker-Smale model is a basic one used to describe flocking, which was put forward in 2007. Motivated by their pioneer work \cite{cucker2007emergent}, Ha et al \cite{ha2009simple} presented a complete analysis for flocking of the Cucker-Smale model by using the Lyapunov functional approach. Then they rigorously justified the mean-field limit from the particle model to the kinetic Cucker-Smale model. Later, Carrillo et al \cite{canizo2011well} give an elegant proof for the mean-field limit by employing the modern theory of optimal transport. In \cite{carrillo2010asymptotic}, they further refined the results in \cite{ha2009simple} and proved the unconditional flocking theorem for the measure-valued solutions to the kinetic Cucker-Smale model. Nowadays, studies of the Cucker-Smale model from particle to kinetic and hydrodynamic description have been launched. We refer the interested readers to \cite{ha2014global, ha2014hydrodynamic, ha2015emergent, jin2015well} and the references therein for the results related to hydrodynamic Cucker-Smale models and the review paper \cite{carrillo2010particle} for the state of the art in this research topic.

However, most mathematical models in this territory are just derived formally. The rigorous limits and  stabilities of many models are still unknown. Even though the stability for the kinetic Cucker-Smale model has been established in measure  space in \cite{canizo2011well} and \cite{ha2009simple}, however the corresponding results in regular function space are still lacking.  As far as we know, there is only existence theory for weak solutions in function space;see \cite{karper2013existence}. The proof of \cite{canizo2011well} and \cite{ha2009simple} are both based on the analysis to the characteristics under the condition that the initial data have compact support. In fact, this method can only deal with the kinetic Cucker-Smale model without noise. Now in the present paper, we will provide a frame that can establish the well-posedness of solutions to the kinetic Cucker-Smale model with or without noise, no matter whether the initial data have compact support or not.

It is well known that we can construct the admissible weak solutions to the hyperbolic conservation laws by using the vanishing viscosity approach; see \cite{bianchini2005vanishing}\cite{hopf1950partial}. Similarly, can we recover  weak and strong solutions to the kinetic Cucker-Smale model by the vanishing noise limit? This is another problem we are concerned with. By using the velocity averaging lemma and subtle mathematical analysis, we give a positive answer to this question. The reader can also refer to \cite{diperna1989global}\cite{diperna1989cauchy} for further application of the velocity averaging lemma in kinetic theory.

The rest of the paper is organised as follows. In section 2, we prove the well-posedness of weak and strong solutions to the kinetic Cucker-Smale model in the Sobolev space. In section 3, we mainly study the kinetic Cucker-Smale model with noise by introducing two weighted Hilbert spaces. While section 4 is devoted to the study of vanishing noise limit. In the last section, we summarise our paper and make a brief comment on it.
\vskip 0.3cm
\noindent\textbf{Notation}. Throughout the paper, we denote by $C$ a general positive constant, which may depend on the initial data. We write $C(\alpha)$ to emphasize that $C$ depends on $\alpha$. While $c(t)$ denotes a general positive function. Both $C$ and $c(t)$ may take different values in different expressions.
%
%
\section{Kinetic Cucker-Smale Model without Noise}
\setcounter{equation}{0}
In this section, we prove the well-posedness of weak and strong solutions to the kinetic Cucker-Smale without noise, i.e., $\sigma=0$ in \eqref{eq-kine}. Then the equation \eqref{eq-kine} reduces to
\begin{equation} \label{eq-kine-wtnois}
     \begin{dcases}
         f_t + \bv \cdot \nabla_{\bx} f+ \nabla_{\bv} \cdot (L[f]f)=0,\\
         f|_{t=0}=f_0(\bx,\bv),
     \end{dcases}
\end{equation}
In reality, the velocities of particles are finite. After taking the mean-field limit, it is natural to suppose that the initial density distribution function has compact velocity support. We study this situation first. The characteristics play a key role in our analysis. We mainly use them to estimate the velocity support of the density distribution function $f(t,\bx, \bv)$. Next we present the definition and results in this section.
\begin{definition}
Let $0 \le f(t,\bx, \bv)\in C([0,+\infty),L^1(\bbr^{2d}))$. $f(t,\bx, \bv)$ is a weak solution to \eqref{eq-kine-wtnois} if
\[
   f_t + \bv \cdot \nabla_{\bx} f+ \nabla_{\bv} \cdot (L[f]f)=0 \quad \text{in} \ \mathcal{D}'([0,+\infty)\times \bbr^{2d}).
\]
We say $f(t,\bx, \bv)$ is a strong solution if $f(t,\bx, \bv)$ is a weak solution and $f(t,\bx, \bv)\in C([0,+\infty),W^{1,1}(\bbr^{2d}))$.
\end{definition}

Denote
\[
  R(t)=\sup\{|\bv|: \ (\bx,\bv) \in \text{supp} f(t,\cdot,\cdot)\}.
\]
We have the following theorems.
\begin{theorem} \label{tm-wek}
Assume $R_0 >0$ and $0 \le f_0(\bx, \bv)\in W^{1,1}(\bbr^{2d}), \text{supp}_{\bv} f_0(\bx,\cdot)\subseteq B(R_0)$. Then \eqref{eq-kine-wtnois} admits a unique weak solution $f(t, \bx,\bv) \in C([0,+\infty),L^1(\bbr^{2d}))$ with its velocity support satisfying
\[
 R(t) \le R_0+\|f_0\|_{L^1(\bbr^{2d})}R_0 t.
\]
Moreover, it holds the following stability estimate
\[
 \sup_{0\le t \le T}\|f(t)-g(t)\|_{L^1(\bbr^{2d})}\le \|f_0-g_0\|_{L^1(\bbr^{2d})}\exp\left(\int_0^T c(t)dt\right), \quad \forall T\ge 0,
\]
where $f(t, \bx,\bv)$ and $g(t, \bx,\bv)$ are weak solutions with initial data $f_0$ and $g_0$ satisfying the above conditions, respectively.
\end{theorem}

\begin{theorem} \label{tm-str}
Assume $R_0 >0$ and $0 \le f_0(\bx, \bv)\in W^{2,1}(\bbr^{2d}), \text{supp}_{\bv} f_0(\bx,\cdot)\subseteq B(R_0)$. Then \eqref{eq-kine-wtnois} admits a unique strong solution $f(t, \bx,\bv) \in C([0,+\infty),W^{1,1}(\bbr^{2d}))$ with its velocity support satisfying
\[
 R(t) \le R_0+\|f_0\|_{L^1(\bbr^{2d})}R_0 t.
\]
Moreover, it holds the following stability estimate
\[
 \sup_{0\le t \le T}\|f(t)-g(t)\|_{W^{1,1}(\bbr^{2d})}\le \|f_0-g_0\|_{W^{1,1}(\bbr^{2d})}\exp\left(\int_0^T c(t)dt\right), \quad \forall T\ge 0,
\]
where $f(t, \bx,\bv)$ and $g(t, \bx,\bv)$ are strong solutions with initial data $f_0$ and $g_0$ satisfying the above conditions, respectively.
\end{theorem}

\begin{remark}
We can refine the estimate of $R(t)$ by using the particle method as in \cite{carrillo2010asymptotic}. In fact, it is uniformly bounded in time.
\end{remark}

\begin{remark}
We can also establish the well-posedness of classical solutions by using the same method and the Sobolev embedding if we improve the regularity of initial data.
\end{remark}

In the following subsection, we derive some a priori estimates that are needed in our proof.

\subsection{A priori estimates}
\begin{lemma} \label{lm-prlm}
Assume $R_0 >0$ and $ f_0(\bx, \bv)\ge 0, \text{supp}_{\bv} f_0(\bx,\cdot)\subseteq B(R_0)$. If $f(t, \bx,\bv)$ is a smooth solution to \eqref{eq-kine-wtnois}, then
  \begin{eqnarray*}
    &&(1)~\|f(t)\|_{L^{1}(\bbr^{2d})}=\|f_0\|_{L^{1}(\bbr^{2d})}, \quad \forall t \ge 0;\\
    &&(2)~\int_{\bbr^{2d}}f(t, \bx,\bv) \bv^2 d \bx d \bv \le \int_{\bbr^{2d}}f_0(\bx,\bv) \bv^2 d \bx d \bv, \quad \forall t \ge 0; \\
    &&(3)~R(t) \le R_0+\|f_0\|_{L^1(\bbr^{2d})}R_0 t, \quad \forall t \ge 0.
  \end{eqnarray*}
\end{lemma}
\begin{proof}
(1) Direct integrating \eqref{eq-kine-wtnois}-1 over $\bbr^{2d}$ gives the conclusion.

(2) Multiplying \eqref{eq-kine-wtnois}-1 by $\bv^2$, we obtain
\begin{equation} \label{eq-eneg}
 \frac{\partial}{\partial t}(f \bv^2) + \bv \cdot \nabla_{\bx} (f \bv^2)+ \nabla_{\bv} \cdot (L[f](f \bv^2))=2fL[f]\cdot \bv.
\end{equation}
Integrating \eqref{eq-eneg} over $[0,t]\times \bbr^{2d}$ yields
\begin{equation} \label{eq-engcsv}
 \begin{gathered}
   \int_{\bbr^{2d}}f(t, \bx,\bv) \bv^2 d \bx d \bv+\int_0^t\int_{\bbr^{4d}} \varphi(|\bx-\by|)f(s, \by,\bv^*)f(s, \bx,\bv)(\bv^*-\bv)^2 d \by d \bv^* d \bx d \bv ds \\=\int_{\bbr^{2d}}f_0(\bx,\bv) \bv^2 d \bx d \bv.
 \end{gathered}
\end{equation}
Then we prove $f \ge 0$ by the method of characteristics. Define $(X(t;\bx_0,\bv_0),V(t;\bx_0,\bv_0))$ as the characteristic issuing from $(\bx_0,\bv_0)$. It satisfies
\begin{equation} \label{eq-charac}
\begin{dcases}
  \frac{d X}{d t}=V, \\
  \frac{d V}{d t}=\int_{\bbr^{2d}} \varphi(|X-\by|)f(t, \by,\bv^*)(\bv^*-V)d \by d \bv^*,
\end{dcases}
\end{equation}
Define
\[
 a(t,\bx):=\int_{\bbr^{2d}} \varphi(|\bx-\by|)f(t, \by,\bv^*) d \by d \bv^*,
\]
\[
 \mathbf{b} (t,\bx):=\int_{\bbr^{2d}} \varphi(|\bx-\by|)f(t, \by,\bv^*) \bv^* d \by d \bv^*.
\]
Solving the equation \eqref{eq-kine-wtnois}, we have
\[
 f(t,X(t;\bx_0,\bv_0),V(t;\bx_0,\bv_0))=f_0(\bx_0,\bv_0)\exp \left(d \int_0^t  a(\tau,X(\tau)) d \tau \right)\ge 0.
\]
Combining with \eqref{eq-engcsv}, we know
\[
 \int_{\bbr^{2d}}f(t, \bx,\bv) \bv^2 d \bx d \bv \le \int_{\bbr^{2d}}f_0(\bx,\bv) \bv^2 d \bx d \bv, \quad \forall t \ge 0.
\]
(3) It follows from the characteristic equation \eqref{eq-charac} that
\begin{equation} \label{eq-exnV}
 V(t)=V_0 e^{-\int_0^t  a(\tau,X(\tau)) d \tau}+e^{-\int_0^t  a(\tau,X(\tau)) d \tau }\int_0^t \mathbf{b}(\tau,X(\tau)) e^{\int_0^{\tau}  a(s,X(s)) ds} d \tau.
\end{equation}
Using the Cauchy's inequality, we have
\begin{equation}
 \begin{aligned}
   |\mathbf{b}(t,\bx)|\le & \left(\int_{\bbr^{2d}} f(t, \by,\bv^*) d \by d \bv^* \right)^{\frac12}\left(\int_{\bbr^{2d}} f(t, \by,\bv^*) |\bv^*|^2 d \by d \bv^* \right)^{\frac12}\\
   \le &\|f_0\|_{L^1(\bbr^{2d})}R_0 .
 \end{aligned}
\end{equation}
From \eqref{eq-exnV}, we deduce that
\[
 R(t) \le R_0+\|f_0\|_{L^1(\bbr^{2d})}R_0 t, \quad \forall t \ge 0.
\]
This completes the proof.
\end{proof}

\begin{lemma} \label{lm-wek}
Assume $R_0 >0$ and $ 0\le f_0(\bx, \bv) \in W^{1,1}(\bbr^{2d}), \text{supp}_{\bv} f_0(\bx,\cdot)\subseteq B(R_0)$. If $f(t, \bx,\bv)$ is a smooth solution to \eqref{eq-kine-wtnois}, then
  \begin{eqnarray*}
    &&(1)~\|f(t)\|_{W^{1,1}(\bbr^{2d})} \le \|f_0\|_{W^{1,1}(\bbr^{2d})} \exp\left(\int_0^t c(\tau)d \tau\right), \quad \forall t \ge 0;\\
    &&(2)~\sup_{0\le t \le T}\|f(t)-g(t)\|_{L^{1}(\bbr^{2d})}\le \|f_0-g_0\|_{L^{1}(\bbr^{2d})}\exp\left(\int_0^T c(t)dt\right), \quad \forall T\ge 0.\\
  \end{eqnarray*}
where $f(t, \bx,\bv)$ and $g(t, \bx,\bv)$ are smooth solutions with initial data $f_0$ and $g_0$ satisfying the above conditions, respectively.
\end{lemma}
\begin{proof}
Applying $\nabla_{\bx}$ to \eqref{eq-kine-wtnois}-1, we have
\begin{equation} \label{eq-wtns-dx}
 (\nabla_{\bx} f)_t +\bv \cdot \nabla_{\bx} (\nabla_{\bx} f)+ \nabla_{\bv} \cdot (L[f]\otimes \nabla_{\bx} f)=-\nabla_{\bx}L[f] \cdot \nabla_{\bv} f- f\nabla_{\bx} \nabla_{\bv} \cdot L[f].
\end{equation}
Multiplying \eqref{eq-wtns-dx} by $(sgn(\partial_{x_1}f),sgn(\partial_{x_2}f),\cdots,sgn(\partial_{x_d}f))$, and integrating the resulting equation over $\bbr^{2d}$, we obtain
\begin{equation} \label{eq-wtns-dxint}
 \frac{d}{dt}\|\nabla_{\bx} f\|_{L^{1}(\bbr^{2d})} \le c(t)\|\nabla_{\bv} f\|_{L^{1}(\bbr^{2d})}+3 \|f_0\|_{L^{1}(\bbr^{2d})} \|f\|_{L^{1}(\bbr^{2d})},
\end{equation}
where we have used Lemma \ref{lm-prlm}.
Applying $\nabla_{\bv}$ to \eqref{eq-kine-wtnois}-1, we obtain
\begin{equation} \label{eq-wtns-dv}
 (\nabla_{\bv} f)_t +\bv \cdot \nabla_{\bx} (\nabla_{\bv} f)+ \nabla_{\bv} \cdot (L[f]\otimes \nabla_{\bv} f)=-\nabla_{\bv}L[f] \cdot \nabla_{\bv} f- \nabla_{\bx} f.
\end{equation}
Similarly, we have
\begin{equation} \label{eq-wtns-dvint}
 \frac{d}{dt}\|\nabla_{\bv} f\|_{L^{1}(\bbr^{2d})} \le \|f_0\|_{L^{1}(\bbr^{2d})} \|\nabla_{\bv} f\|_{L^{1}(\bbr^{2d})}+\|\nabla_{\bx} f\|_{L^{1}(\bbr^{2d})}.
\end{equation}
Adding \eqref{eq-wtns-dxint} to \eqref{eq-wtns-dvint} and using the fact that $ \frac{d}{dt}\|f\|_{L^{1}(\bbr^{2d})}=0$, we arrive at
\begin{equation} \label{eq-wtns-dvint}
 \frac{d}{dt}\| f(t)\|_{W^{1,1}(\bbr^{2d})} \le c(t) \|f(t)\|_{W^{1,1}(\bbr^{2d})}.
\end{equation}
Following from the Gronwall's inequality, we deduce that
\begin{equation} \label{eq-w1nm}
 \|f(t)\|_{W^{1,1}(\bbr^{2d})} \le \|f_0\|_{W^{1,1}(\bbr^{2d})} \exp\left(\int_0^t c(\tau)d \tau\right), \quad \forall t \ge 0.
\end{equation}
(2) For two smooth solutions $f(t,\bx,\bv)$ and $g(t,\bx,\bv)$ with initial data $f_0$ and $g_0$, respectively. We define
\[
 \bar{h}:=f-g
\]
It follows from the equation \eqref{eq-kine-wtnois} that
\begin{equation} \label{eq-kine-dfr}
 \bar{h}_t + \bv \cdot \nabla_{\bx} \bar{h}+ \nabla_{\bv} \cdot (L[f]\bar{h})=-g \nabla_{\bv} \cdot L[\bar{h}]-L[\bar{h}] \cdot \nabla_{\bv}g.
\end{equation}
Multiplying \eqref{eq-kine-dfr} by $sgn(\bar{h})$ and integrating the resulting equation over $\bbr^{2d}$, we obtain
\begin{equation} \label{eq-kine-dfrnm}
 \frac{d}{dt}\|\bar{h}\|_{L^{1}(\bbr^{2d})} \le C \|g\|_{L^{1}(\bbr^{2d})}\| \bar{h}\|_{L^{1}(\bbr^{2d})}+c(t) \|\nabla_{\bv}g\|_{L^{1}(\bbr^{2d})} \| \bar{h}\|_{L^{1}(\bbr^{2d})}.
\end{equation}
Combining with \eqref{eq-w1nm} and solving the above Gronwall's inequality, we have
\begin{equation} \label{eq-wstb}
\sup_{0\le t \le T}\|f(t)-g(t)\|_{L^{1}(\bbr^{2d})}\le \|f_0-g_0\|_{L^{1}(\bbr^{2d})}\exp\left(\int_0^T c(t)dt\right), \quad \forall T\ge 0.
\end{equation}
\end{proof}

\begin{lemma} \label{lm-str}
Assume $R_0 >0$ and $ 0\le f_0(\bx, \bv) \in W^{2,1}(\bbr^{2d}), \text{supp}_{\bv} f_0(\bx,\cdot)\subseteq B(R_0)$. If $f(t, \bx,\bv)$ is a smooth solution to \eqref{eq-kine-wtnois}, then
  \begin{eqnarray*}
    &&(1)~\|f(t)\|_{W^{2,1}(\bbr^{2d})} \le \|f_0\|_{W^{2,1}(\bbr^{2d})} \exp\left(\int_0^t c(\tau)d \tau\right), \quad \forall t \ge 0;\\
    &&(2)~\sup_{0\le t \le T}\|f(t)-g(t)\|_{W^{1,1}(\bbr^{2d})}\le \|f_0-g_0\|_{W^{1,1}(\bbr^{2d})}\exp\left(\int_0^T c(t)dt\right), \quad \forall T\ge 0.
  \end{eqnarray*}
where $f(t, \bx,\bv)$ and $g(t, \bx,\bv)$ are smooth solutions with initial data $f_0$ and $g_0$ satisfying the above conditions, respectively.
\end{lemma}
\begin{proof}
Based on Lemma \ref{lm-wek}, we only need to estimate the second-order derivatives. Applying $\partial_{x_i}\partial_{x_j}$ to \eqref{eq-kine-wtnois}-1, we obtain
\begin{equation} \label{eq-kine2dx}
 \begin{aligned}
   (\partial_{x_i}\partial_{x_j}f)_t &+ \bv \cdot \nabla_{\bx} (\partial_{x_i}\partial_{x_j}f)+ \nabla_{\bv} \cdot (L[f] \partial_{x_i}\partial_{x_j}f)
   \\=&-\partial_{x_j}L[f] \cdot \nabla_{\bv} \partial_{x_i} f- \partial_{x_i} f \partial_{x_j} \nabla_{\bv} \cdot L[f]-\partial_{x_i} \partial_{x_j} L[f]\cdot \nabla_{\bv} f
   \\&-\partial_{x_i}  L[f]\cdot \nabla_{\bv} \partial_{x_j}f -\partial_{x_j} f \partial_{x_i} \nabla_{\bv} \cdot L[f]- f \partial_{x_i} \partial_{x_j}\nabla_{\bv} \cdot L[f].
 \end{aligned}
\end{equation}
Multiplying \eqref{eq-kine2dx} by $sgn(\partial_{x_i}\partial_{x_j}f)$ and integrating the resulting equation over $\bbr^{2d}$, we get
\begin{equation} \label{eq-kine2dxit}
 \begin{aligned}
  \frac{d}{dt} \|\partial_{x_i}\partial_{x_j}f\|_{L^1(\bbr^{2d})}\le &c(t)\|f\|_{L^1(\bbr^{2d})}
  \|\nabla_{\bv}\partial_{x_i}f\|_{L^1(\bbr^{2d})}+C\|\partial_{x_i}f\|_{L^1(\bbr^{2d})}\|f\|_{L^1(\bbr^{2d})}\\
  &+c(t)\|f\|_{L^1(\bbr^{2d})}  \|\nabla_{\bv}f\|_{L^1(\bbr^{2d})}
  +c(t)\|f\|_{L^1(\bbr^{2d})}\|\nabla_{\bv}\partial_{x_j}f\|_{L^1(\bbr^{2d})}\\
  &+C\|\partial_{x_j}f\|_{L^1(\bbr^{2d})}\|f\|_{L^1(\bbr^{2d})}+C\|f\|_{L^1(\bbr^{2d})}\|f\|_{L^1(\bbr^{2d})}
 \end{aligned}
\end{equation}
Applying $\partial_{x_i}\partial_{v_j}$ to \eqref{eq-kine-wtnois}-1, we obtain
\begin{equation} \label{eq-kine2dxv}
 \begin{aligned}
   (\partial_{x_i}\partial_{v_j}f)_t &+ \bv \cdot \nabla_{\bx} (\partial_{x_i}\partial_{v_j}f)+ \nabla_{\bv} \cdot (L[f] \partial_{x_i}\partial_{v_j}f)
   \\=&-\partial_{v_j}L[f] \cdot \nabla_{\bv} \partial_{x_i} f-\partial_{x_i} \partial_{v_j} L[f]\cdot \nabla_{\bv} f
   \\&-\partial_{x_i}  L[f]\cdot \nabla_{\bv} \partial_{v_j}f -\partial_{v_j} f \partial_{x_i} \nabla_{\bv} \cdot L[f]-  \partial_{x_i} \partial_{x_j}f.
 \end{aligned}
\end{equation}
Using the method as above, we get
\begin{equation} \label{eq-kine2dxvit}
 \begin{aligned}
  \frac{d}{dt} \|\partial_{x_i}\partial_{v_j}f\|_{L^1(\bbr^{2d})}\le &\|f\|_{L^1(\bbr^{2d})}
  \|\partial_{v_j}\partial_{x_i}f\|_{L^1(\bbr^{2d})}\\
  &+\|f\|_{L^1(\bbr^{2d})}  \|\partial_{v_j}f\|_{L^1(\bbr^{2d})}
  +c(t)\|f\|_{L^1(\bbr^{2d})}\|\nabla_{\bv}\partial_{v_j}f\|_{L^1(\bbr^{2d})}\\
  &+C\|\partial_{v_j}f\|_{L^1(\bbr^{2d})}\|f\|_{L^1(\bbr^{2d})}+
  \|\partial_{x_i}\partial_{x_j}f\|_{L^1(\bbr^{2d})}.
 \end{aligned}
\end{equation}
Applying $\partial_{v_i}\partial_{v_j}$ to \eqref{eq-kine-wtnois}-1, we obtain
\begin{equation} \label{eq-kine2dv}
 \begin{aligned}
   (\partial_{v_i}\partial_{v_j}f)_t &+ \bv \cdot \nabla_{\bx} (\partial_{v_i}\partial_{v_j}f)+ \nabla_{\bv} \cdot (L[f] \partial_{v_i}\partial_{v_j}f)
   \\=&-\partial_{v_j}L[f] \cdot \nabla_{\bv} \partial_{v_i} f
   \\&-\partial_{v_i}  L[f]\cdot \nabla_{\bv} \partial_{v_j}f - \partial_{x_i} \partial_{v_j}f.
 \end{aligned}
\end{equation}
Similarly, we have
\begin{equation} \label{eq-kine2dvit}
 \begin{aligned}
  \frac{d}{dt} \|\partial_{v_i}\partial_{v_j}f\|_{L^1(\bbr^{2d})}\le &\|f\|_{L^1(\bbr^{2d})}
  \|\partial_{v_j}\partial_{v_i}f\|_{L^1(\bbr^{2d})}\\
  &+\|f\|_{L^1(\bbr^{2d})}  \|\partial_{v_i}\partial_{v_j}f\|_{L^1(\bbr^{2d})}
  + \|\partial_{x_i}\partial_{v_j}f\|_{L^1(\bbr^{2d})}.
 \end{aligned}
\end{equation}
Adding \eqref{eq-kine2dxit}, \eqref{eq-kine2dxvit}, \eqref{eq-kine2dvit} together, summing over all $1\le i,j \le d$ and then combining with \eqref{eq-wtns-dvint}, we arrive at
\begin{equation} \label{eq-wtns-dv2int}
 \frac{d}{dt}\| f(t)\|_{W^{2,1}(\bbr^{2d})} \le c(t) \|f(t)\|_{W^{2,1}(\bbr^{2d})}.
\end{equation}
It follows from the Gronwall's inequality that
\begin{equation} \label{eq-w2nm}
 \|f(t)\|_{W^{2,1}(\bbr^{2d})} \le \|f_0\|_{W^{2,1}(\bbr^{2d})} \exp\left(\int_0^t c(\tau)d \tau\right), \quad \forall t \ge 0.
\end{equation}
(2) For two smooth solutions $f(t,\bx,\bv)$ and $g(t,\bx,\bv)$ with initial data $f_0$ and $g_0$ satisfying the above initial conditions, respectively. We define
\[
 \bar{h}:=f-g, \quad \overline{F}:=\nabla_{\bx}f-\nabla_{\bx}g \quad \text{and} \quad \overline{G}:=\nabla_{\bv}f-\nabla_{\bv}g.
\]
It follows from \eqref{eq-wtns-dx} that
\begin{equation} \label{eq-kine-dfr-dx}
 \begin{aligned}
   \overline{F}_t &+ \bv \cdot \nabla_{\bx} \overline{F}+ \nabla_{\bv} \cdot (L[f]\otimes \overline{F})\\
   =&-\nabla_{\bv} \cdot (L[\bar{h}]\otimes \nabla_{\bx}g)-\nabla_{\bx}L[\bar{h}]\cdot \nabla_{\bv}f\\
   &-\nabla_{\bx}L[g]\cdot \overline{G}- \bar{h}\nabla_{\bx}\nabla_{\bv} \cdot L[f]-g \nabla_{\bx}\nabla_{\bv} \cdot L[\bar{h}].
 \end{aligned}
\end{equation}
Multiplying \eqref{eq-kine-dfr-dx} by $sgn(\overline{F})$ with each component being the signal function of the corresponding one of $\overline{F}$, and then integrating the resulting equation over $\bbr^{2d}$, we obtain
\begin{equation} \label{eq-kine-dfr-dxit}
 \begin{aligned}
  \frac{d}{dt} \|\overline{F}\|_{L^1(\bbr^{2d})}\le & C\|\bar{h}\|_{L^1(\bbr^{2d})}
  \|\nabla_{\bx}g\|_{L^1(\bbr^{2d})}+c(t)\|\nabla_{\bx}\nabla_{\bv}g\|_{L^1(\bbr^{2d})}
  \|\bar{h}\|_{L^1(\bbr^{2d})}\\
  &+c(t)\|\nabla_{\bv}f\|_{L^1(\bbr^{2d})}  \|\bar{h}\|_{L^1(\bbr^{2d})}
  +c(t)\|g\|_{L^1(\bbr^{2d})}\|\overline{G}\|_{L^1(\bbr^{2d})}\\
  &+C\|f\|_{L^1(\bbr^{2d})}\|\bar{h}\|_{L^1(\bbr^{2d})}+C\|g\|_{L^1(\bbr^{2d})}\|\bar{h}\|_{L^1(\bbr^{2d})}.
 \end{aligned}
\end{equation}
From \eqref{eq-wtns-dv}, we deduce that
\begin{equation} \label{eq-kine-dfr-dv}
 \begin{aligned}
   &\overline{G}_t + \bv \cdot \nabla_{\bx} \overline{G}+ \nabla_{\bv} \cdot (L[f]\otimes \overline{G})\\
   =&-\nabla_{\bv} \cdot (L[\bar{h}]\otimes \nabla_{\bv}g)
   -\nabla_{\bv}L[\bar{h}]\cdot \nabla_{\bv} f- \nabla_{\bv}L[g] \cdot \overline{G}-\overline{F}.
 \end{aligned}
\end{equation}
Similarly, we have
\begin{equation} \label{eq-kine-dfr-dvit}
 \begin{aligned}
  \frac{d}{dt} \|\overline{G}\|_{L^1(\bbr^{2d})}\le & C\|\nabla_{\bv}g\|_{L^1(\bbr^{2d})}\|\bar{h}\|_{L^1(\bbr^{2d})}
  +c(t)\|\nabla_{\bv}^2 g\|_{L^1(\bbr^{2d})}
  \|\bar{h}\|_{L^1(\bbr^{2d})}\\
  &+\|\nabla_{\bv}f\|_{L^1(\bbr^{2d})}  \|\bar{h}\|_{L^1(\bbr^{2d})}
  +\|g\|_{L^1(\bbr^{2d})}\|\overline{G}\|_{L^1(\bbr^{2d})}+
  \|\overline{F}\|_{L^1(\bbr^{2d})}.
 \end{aligned}
\end{equation}
Adding \eqref{eq-kine-dfr-dxit} to \eqref{eq-kine-dfr-dvit} and combining with \eqref{eq-kine-dfrnm}, \eqref{eq-w2nm}, we arrive at
\begin{equation} \label{eq-kine-dfr1nm}
 \frac{d}{dt}\|f-g\|_{W^{1,1}(\bbr^{2d})} \le c(t)\| f-g\|_{W^{1,1}(\bbr^{2d})}.
\end{equation}
Solving the above Gronwall's inequality, we have
\begin{equation} \label{eq-wstb}
\sup_{0\le t \le T}\|f(t)-g(t)\|_{W^{1,1}(\bbr^{2d})}\le \|f_0-g_0\|_{W^{1,1}(\bbr^{2d})}\exp\left(\int_0^T c(t)dt\right), \quad \forall T\ge 0.
\end{equation}
This completes the proof.
\end{proof}

Higher order estimates can also be obtained with the same method. Next we present the proof of Theorem \ref{tm-wek} and Theorem \ref{tm-str}.
\subsection{Proof of Theorem \ref{tm-wek} and Theorem \ref{tm-str}}
We first mollify the initial data by convolution,i.e.,$$f_0^{\e}(\bx, \bv)=f_0 \ast j_{\e}(\bx, \bv),$$ where $j_{\e}$ is the standard mollifier. Using the contraction principle, we can obtain the local smooth solution by the standard procedure. Combining with the a priori estimate in Lemma \ref{lm-wek}(1), one can extend the local smooth solution to be global-in-time.

Then using the stability estimate in Lemma \ref{lm-wek}(2), we infer that
\begin{equation} \label{eq-wkcov}
\sup_{0\le t \le T}\|f^{\e_i}(t)-f^{\e_j}(t)\|_{L^{1}(\bbr^{2d})}\le \|f_0^{\e_i}-f_0^{\e_j}\|_{L^{1}(\bbr^{2d})}\exp\left(\int_0^T c(t)dt\right), \quad \forall T\ge 0,
\end{equation}
where $f^{\e_i}(t,\bx, \bv)$ and $f^{\e_j}(t,\bx, \bv)$ are smooth solution with initial data $f_0^{\e_i}$ and $f_0^{\e_j}$, respectively.
From \eqref{eq-wkcov}, we know there exists $f(t,\bx, \bv) \in C([0,T],L^1(\bbr^{2d}))$ such that
\[
 f^{\e_i}(t,\bx, \bv) \to f(t,\bx, \bv) \quad \text{in $C([0,T],L^1(\bbr^{2d}))$, as $\e_i \to 0$}.
\]
Due to the arbitrariness of $T$, we know $f(t,\bx, \bv) \in C([0,+\infty),L^1(\bbr^{2d}))$. It is easy to see that $f(t,\bx, \bv)$ is a weak solution to \eqref{eq-kine-wtnois}.
Take smooth initial data $f_0^{\e_i}$ and $g_0^{\e_i}$. We also have
\begin{equation} \label{eq-wkstacov}
\sup_{0\le t \le T}\|f^{\e_i}(t)-g^{\e_i}(t)\|_{L^{1}(\bbr^{2d})}\le \|f_0^{\e_i}-g_0^{\e_i}\|_{L^{1}(\bbr^{2d})}\exp\left(\int_0^T c(t)dt\right), \quad \forall T\ge 0,
\end{equation}
Letting $\e_i \to 0$, we obtain the stability estimate for weak solutions to \eqref{eq-kine-wtnois}, which  amounts to uniqueness of the weak solution.

Theorem \ref{tm-str} can be proved in the same way. We omit its proof for brevity. Thus we complete the proof.
%
%
\section{Kinetic Cucker-Smale Model with Noise}
\setcounter{equation}{0}
In this section, we study the kinetic Cucker-Smale model with noise, i.e.,
\begin{equation} \label{eq-kine-nois}
   \begin{dcases}
      f_t + \bv \cdot \nabla_{\bx} f+ \nabla_{\bv} \cdot (L[f]f)=\sigma \Delta_{\bv}f, \quad 0<\sigma \le 1\\
      f|_{t=0}=f_0(\bx,\bv),
   \end{dcases}
\end{equation}
Unlike \eqref{eq-kine-wtnois}, the velocity support of the solution to this equation may be unbounded, even if the the initial data have compact velocity support. So the method in section 2 is not valid. In order to circumvent this difficulty, we introduce two weighted Hilbert spaces to establish the well-posedness of weak and strong solutions to \eqref{eq-kine-nois}.
Define
\[
 \|f\|_{L^2(\omega)}=\left( \int_{\bbr^{2d}} f^2(\bx, \bv)\omega(\bx, \bv) d \bx d\bv \right)^{\frac12},
\]
\[
\|f\|_{L^2(\nu)}=\left(\int_{\bbr^{2d}} f^2(\bx, \bv)\nu(\bv) d \bx d\bv \right)^{\frac12},
\]
\[
 X=\{f:\  f \in L^2(\omega), \ \nabla_{\bx}f \in L^2(\nu), \ \nabla_{\bv}f \in L^2(\bbr^{2d}) \},
\]
\[
 \|f\|_X^2=\|f\|_{L^2(\omega)}^2+\|\nabla_{\bx}f\|_{L^2(\nu)}^2+\|\nabla_{\bv}f\|_{L^2(\bbr^{2d})}^2,
\]
where $\omega(\bx, \bv)=(1+\bv^2)(1+\bx^2+\bv^2)^{\alpha}, \alpha >3$ and  $\nu(\bv)=1+\bv^2$.
Next we present the definition and results in this section.
\begin{definition} \label{def-ws-nois}
Let $f(t,\bx, \bv)\in C([0,+\infty),L^2(\bbr^{2d}))$. $f(t,\bx, \bv)$ is a weak solution to \eqref{eq-kine-nois} if
\[
   f_t + \bv \cdot \nabla_{\bx} f+ \nabla_{\bv} \cdot (L[f]f)=\sigma \Delta_{\bv}f, \quad \text{in} \ \mathcal{D}'([0,+\infty)\times \bbr^{2d}).
\]
We say $f(t,\bx, \bv)$ is a strong solution if $f(t,\bx, \bv)$ is a weak solution and $f(t,\bx, \bv)\in C([0,+\infty),H^{1}(\bbr^{2d})) \cap L^2((0,T) \times \bbr^{d},H^2(\bbr_{\bv}^{d})), \forall T \ge 0$.
\end{definition}
\begin{remark}
Since the strong solution means that a solution satisfies the equation almost everywhere, thus $f(t,\bx, \bv)$ is still a strong solution to \eqref{eq-kine-wtnois} if $f(t,\bx, \bv)$ is a weak solution and $f(t,\bx, \bv)\in C([0,+\infty),H^{1}(\bbr^{2d}))$.
\end{remark}
\begin{theorem} \label{tm-wek-nois}
Assume $(1+\bv^2)^{\frac12}f_0(\bx, \bv)\in  L^2(\omega)$. Then \eqref{eq-kine-nois} admits a unique weak solution $f(t, \bx,\bv) \in C([0,+\infty),L^2(\omega))$.
Moreover, it hold that
\begin{eqnarray*}
    &&(1)~\|(1+\bv^2)^{\frac12}f(t)\|_{L^2(\omega)}^2+\sigma \int_0^t \|(1+\bv^2)^{\frac12}\nabla_{\bv}f(\tau)\|_{L^2(\omega)}^2 d \tau \\
     &&\qquad \le \|(1+\bv^2)^{\frac12}f_0\|_{L^2(\omega)}^2 \exp\left(\int_0^t c(\tau)d \tau\right), \quad a.e.\ t \ge 0;\\
    &&(2)~\sup_{0\le t \le T}\|f(t)-g(t)\|_{L^2(\omega)}\le \|f_0-g_0\|_{L^2(\omega)}\exp\left(\int_0^T c(t)dt\right), \quad \forall T\ge 0,
  \end{eqnarray*}
where $f(t, \bx,\bv)$ and $g(t, \bx,\bv)$ are weak solutions with initial data $f_0$ and $g_0$ satisfying the above conditions, respectively.
\end{theorem}

\begin{theorem} \label{tm-str-nois}
Assume $(1+\bv^2)^{\frac12}f_0(\bx, \bv)\in  X$. Then \eqref{eq-kine-nois} admits a unique strong solution $f(t,\bx, \bv)\in C([0,+\infty), X)$.
Moreover, it hold that
\begin{eqnarray*}
    &&(1)~\|(1+\bv^2)^{\frac12}f(t)\|_{X}^2+\sigma \int_0^t \|(1+\bv^2)^{\frac12}\nabla_{\bv}f(\tau)\|_{X}^2 d \tau \\
     &&\qquad \le \|(1+\bv^2)^{\frac12}f_0\|_{X}^2 \exp\left(\int_0^t c(\tau)d \tau\right), \quad a.e.\ t \ge 0;\\
    &&(2)~\sup_{0\le t \le T}\|f(t)-g(t)\|_{X}\le \|f_0-g_0\|_{X} \exp\left(\int_0^T c(t)dt\right), \quad \forall T\ge 0,
  \end{eqnarray*}
where $f(t, \bx,\bv)$ and $g(t, \bx,\bv)$ are strong solutions with initial data $f_0$ and $g_0$ satisfying the above conditions, respectively.
\end{theorem}

\begin{remark}\label{rm-cla}
We can also establish the well-posedness of classical solutions to \eqref{eq-kine-nois} by using the same method and the Sobolev embedding if we improve the regularity of initial data.
\end{remark}

In the following subsection, we derive some a priori estimates that are needed in our proof.

\subsection{A priori estimates}
\begin{lemma} \label{lm-prlm-nois}
Assume $(1+\bv^2)^{\frac12}f_0(\bx, \bv)\in  L^2(\omega)$. If $f(t, \bx,\bv)$ is a smooth solution to \eqref{eq-kine-nois}, then
  \begin{eqnarray*}
    &&(1)~\|f(t)\|_{L^{1}(\bbr^{2d})}=\|f_0\|_{L^{1}(\bbr^{2d})}, \quad \forall t \ge 0;\\
    &&(2)~\|(1+\bv^2)^{\frac12} f(t)\|_{L^{1}(\bbr^{2d})} \le \|(1+\bv^2)^{\frac12} f_0\|_{L^{1}(\bbr^{2d})}e^{Ct}, \quad \forall t \ge 0; \\
    &&(3)~\|(1+\bv^2)^{\frac12}f(t)\|_{L^2(\omega)}^2+\sigma \int_0^t \|(1+\bv^2)^{\frac12}\nabla_{\bv}f(\tau)\|_{L^2(\omega)}^2 d \tau \\
     &&\qquad \le \|(1+\bv^2)^{\frac12}f_0\|_{L^2(\omega)}^2 \exp\left(\int_0^t c(\tau)d \tau\right), \quad \forall t \ge 0.
  \end{eqnarray*}
\end{lemma}
\begin{proof}
(1) Since $(1+\bv^2)^{\frac12}f_0(\bx, \bv)\in  L^2(\omega)$, it is easy to see that
\begin{equation} \label{eq-lone}
 \|(1+\bv^2)^{\frac12} f_0\|_{L^{1}(\bbr^{2d})} \le C \|(1+\bv^2)^{\frac12}f_0\|_{L^2(\omega)}.
\end{equation}
Direct integrating \eqref{eq-kine-nois} over $[0,t]\times \bbr^{2d}$ yields
\[
 \|f(t)\|_{L^{1}(\bbr^{2d})}=\|f_0\|_{L^{1}(\bbr^{2d})}, \quad \forall t \ge 0.
\]
(2) Multiplying \eqref{eq-kine-nois} by $(1+\bv^2)^{\frac12}$, we deduce that
\begin{equation} \label{eq-lonef}
  \begin{gathered}
    ((1+\bv^2)^{\frac12}f)_t + \bv \cdot \nabla_{\bx} ((1+\bv^2)^{\frac12}f)+ \nabla_{\bv} \cdot (L[f](1+\bv^2)^{\frac12}f)\\
    =\sigma (1+\bv^2)^{\frac12}\Delta_{\bv}f+fL[f]\cdot \nabla_{\bv}(1+\bv^2)^{\frac12}.
  \end{gathered}
\end{equation}
Integrating \eqref{eq-lonef} over $\bbr^{2d}$ and performing integration by parts, we obtain
\begin{equation} \label{eq-lonefit}
 \begin{aligned}
  \frac{d}{dt} \|(1+\bv^2)^{\frac12}f\|_{L^1(\bbr^{2d})}\le & \sigma \|f\|_{L^1(\bbr^{2d})}+\|f \bv\|_{L^1(\bbr^{2d})}\|f\|_{L^1(\bbr^{2d})}+\|f\|_{L^1(\bbr^{2d})}
  \|f \bv\|_{L^1(\bbr^{2d})}\\
  \le &(2\|f_0\|_{L^1(\bbr^{2d})}+1)\|(1+\bv^2)^{\frac12}f\|_{L^1(\bbr^{2d})}.
 \end{aligned}
\end{equation}
It follows from the Gronwall's inequality that
\begin{equation} \label{eq-lonefest}
 \|(1+\bv^2)^{\frac12} f(t)\|_{L^{1}(\bbr^{2d})} \le \|(1+\bv^2)^{\frac12} f_0\|_{L^{1}(\bbr^{2d})}e^{Ct}, \quad \forall t \ge 0.
\end{equation}
(3) Multiplying \eqref{eq-kine-nois} by $2f(1+\bv^2)\omega$, we get
\begin{equation} \label{eq-l2fwt}
 \begin{aligned}
   &((1+\bv^2)\omega f^2)_t + \bv \cdot \nabla_{\bx} ((1+\bv^2)\omega f^2)+ \nabla_{\bv} \cdot (L[f](1+\bv^2)\omega f^2)\\
    =&f^2 L[f] \cdot \nabla_{\bv}((1+\bv^2)\omega) +2\sigma (1+\bv^2)\omega f \Delta_{\bv}f-(1+\bv^2)\omega f^2\nabla_{\bv} \cdot L[f].
 \end{aligned}
\end{equation}
Integrating \eqref{eq-l2fwt} over $\bbr^{2d}$ and using integration by parts, we deduce that
\begin{equation} \label{eq-l2fwtint}
 \begin{aligned}
  &\frac{d}{dt} \|(1+\bv^2)^{\frac12}f\|_{L^2(\omega)}^2+ \sigma \|(1+\bv^2)^{\frac12} \nabla_{\bv}f\|_{L^2(\omega)}^2\\
  \le &C\|(1+\bv^2)^{\frac12}f\|_{L^1(\bbr^{2d})}\|(1+\bv^2)^{\frac12}f\|_{L^2(\omega)}^2+ \sigma \|(1+\bv^2)^{\frac12}f\|_{L^2(\omega)}^2 \\&+C\|f\|_{L^1(\bbr^{2d})}\|(1+\bv^2)^{\frac12}f\|_{L^2(\omega)}^2.
 \end{aligned}
\end{equation}
Combining with \eqref{eq-lonefest} and solving the above Gronwall's inequality yield
\begin{equation} \label{eq-ltwof}
  \begin{gathered}
    \|(1+\bv^2)^{\frac12}f(t)\|_{L^2(\omega)}^2+\sigma \int_0^t \|(1+\bv^2)^{\frac12}\nabla_{\bv}f(\tau)\|_{L^2(\omega)}^2 d \tau \\
     \le \|(1+\bv^2)^{\frac12}f_0\|_{L^2(\omega)}^2 \exp\left(\int_0^t c(\tau)d \tau\right), \quad \forall t \ge 0.
  \end{gathered}
\end{equation}
This completes the proof.
\end{proof}

\begin{lemma} \label{lm-stb-nois}
If $f(t, \bx,\bv)$ and $g(t, \bx,\bv)$ are two smooth solutions with initial data $f_0$ and $g_0$ satisfying the condition in Lemma \ref{lm-prlm-nois}, respectively, then
\[
 \sup_{0\le t \le T}\|f(t)-g(t)\|_{L^2(\omega)}\le \|f_0-g_0\|_{L^2(\omega)}\exp\left(\int_0^T c(t)dt\right), \quad \forall T\ge 0.
\]
\end{lemma}
\begin{proof}
Define $\bar{h}:=f-g$. It follows from the equation \eqref{eq-kine-nois} that
\begin{equation} \label{eq-kine-dfr-nois}
 \bar{h}_t + \bv \cdot \nabla_{\bx} \bar{h}+ \nabla_{\bv} \cdot (L[f]\bar{h}+L[\bar{h}]g)=\sigma \Delta_{\bv}\bar{h}.
\end{equation}
Multiplying \eqref{eq-kine-dfr-nois} by $2 \bar{h} \omega$, we deduce that
\begin{equation} \label{eq-dfrwt}
 \begin{aligned}
   &(\bar{h}^2 \omega)_t + \bv \cdot \nabla_{\bx} (\bar{h}^2 \omega)+ \nabla_{\bv} \cdot (L[f]\bar{h}^2 \omega)\\
   =&2\sigma \omega \bar{h}\Delta_{\bv}\bar{h}+\bar{h}^2 \bv \cdot \nabla_{\bx}\omega-\bar{h}^2 \omega \nabla_{\bv} \cdot L[f]\\&+\bar{h}^2 L[f] \cdot \nabla_{\bv}\omega-2 \bar{h} \omega \nabla_{\bv} \cdot (L[\bar{h}]g).
 \end{aligned}
\end{equation}
Integrating \eqref{eq-dfrwt} over $\bbr^{2d}$, we obtain
\begin{equation} \label{eq-dfrwtint}
 \begin{aligned}
   \frac{d}{dt}\|\bar{h}\|_{L^2(\omega)}^2
   =&2\sigma \int_{\bbr^{2d}}\omega \bar{h}\Delta_{\bv}\bar{h} d\bx d\bv+\int_{\bbr^{2d}}\bar{h}^2 \bv \cdot \nabla_{\bx}\omega d\bx d\bv- \int_{\bbr^{2d}}\bar{h}^2 \omega \nabla_{\bv} \cdot L[f] d\bx d \bv\\ &+\int_{\bbr^{2d}}\bar{h}^2 L[f] \cdot \nabla_{\bv}\omega d\bx d\bv-2\int_{\bbr^{2d}} \bar{h} \omega \nabla_{\bv} \cdot (L[\bar{h}]g) d\bx d\bv\\
   =&\sum_{k=1}^{5}I_k.
 \end{aligned}
\end{equation}
We estimate each $I_k(1\le k \le 5)$ as follows.
\begin{equation*}
 \begin{aligned}
  &I_1=-2 \sigma \|\nabla_{\bv}\bar{h}\|_{L^2(\omega)}^2 +\sigma \int_{\bbr^{2d}}\bar{h}^2 \Delta_{\bv} \omega d\bx d \bv,\\
  &I_2 \le C \int_{\bbr^{2d}} \bar{h}^2 \omega d\bx d\bv=C\|\bar{h}\|_{L^2(\omega)}^2,\\
  &I_3 \le C\|f\|_{L^1(\bbr^{2d})}\|\bar{h}\|_{L^2(\omega)}^2,\\
  &I_4 \le C\|(1+\bv^2)^{\frac12}f\|_{L^1(\bbr^{2d})}\|\bar{h}\|_{L^2(\omega)}^2,\\
  &I_5 =2  \int_{\bbr^{2d}}g L[\bar{h}] \cdot (\nabla_{\bv}\bar{h} \omega+\nabla_{\bv}\omega \bar{h})d\bx d\bv\\
  &\quad \le C\|\bar{h}\|_{L^2(\omega)}\|(1+\bv^2)^{\frac12}g\|_{L^2(\omega)}\|\nabla_{\bv}\bar{h}\|_{L^2(\omega)}
  + C\|g\|_{L^2(\omega)}\|\bar{h}\|_{L^2(\omega)}^2\\
 &\quad \le \sigma \|\nabla_{\bv}\bar{h}\|_{L^2(\omega)}^2+C\|(1+\bv^2)^{\frac12}g\|_{L^2(\omega)}^2\|\bar{h}\|_{L^2(\omega)}^2
  +C\|g\|_{L^2(\omega)}\|\bar{h}\|_{L^2(\omega)}^2
 \end{aligned}
\end{equation*}
Substituting these estimates into \eqref{eq-dfrwtint}, we obtain
\begin{equation} \label{eq-dfrwtinmt}
 \begin{aligned}
   &\frac{d}{dt}\|\bar{h}\|_{L^2(\omega)}^2+\sigma \|\nabla_{\bv}\bar{h}\|_{L^2(\omega)}^2\\
   \le &C\left(1+\|(1+\bv^2)^{\frac12}f\|_{L^1(\bbr^{2d})}+\|(1+\bv^2)^{\frac12}g\|_{L^2(\omega)}^2 \right)
   \|\bar{h}\|_{L^2(\omega)}^2.
 \end{aligned}
\end{equation}
Combining with Lemma \ref{lm-prlm-nois} and solving the above Gronwall's inequality give
\[
 \sup_{0\le t \le T}\|f(t)-g(t)\|_{L^2(\omega)}\le \|f_0-g_0\|_{L^2(\omega)}\exp\left(\int_0^T c(t)dt\right), \quad \forall T\ge 0.
\]
This completes the proof.
\end{proof}

\begin{lemma} \label{lm-str-nois}
Assume $(1+\bv^2)^{\frac12}f_0(\bx, \bv)\in  X$. If $f(t,\bx, \bv)$ is a smooth solution to \eqref{eq-kine-nois}, then
\begin{eqnarray*}
    &&(1)~\|(1+\bv^2)^{\frac12}f(t)\|_{X}^2+\sigma \int_0^t \|(1+\bv^2)^{\frac12}\nabla_{\bv}f(\tau)\|_{X}^2 d \tau \\
     &&\qquad \le \|(1+\bv^2)^{\frac12}f_0\|_{X}^2 \exp\left(\int_0^t c(\tau)d \tau\right), \quad \forall t \ge 0;\\
    &&(2)~\sup_{0\le t \le T}\|f(t)-g(t)\|_{X}\le \|f_0-g_0\|_{X} \exp\left(\int_0^T c(t)dt\right), \quad \forall T\ge 0,
  \end{eqnarray*}
where $f(t, \bx,\bv)$ and $g(t, \bx,\bv)$ are smooth solutions with initial data $f_0$ and $g_0$ satisfying the above conditions, respectively.
\end{lemma}
\begin{proof}
Based on Lemma \ref{lm-prlm-nois} and \ref{lm-stb-nois}, we only need to estimate the first order derivatives.
Applying $\nabla_{\bx}$ to \eqref{eq-kine-nois} yields
\begin{equation} \label{eq-ns-dx}
 (\nabla_{\bx} f)_t +\bv \cdot \nabla_{\bx} (\nabla_{\bx} f)+ \nabla_{\bv} \cdot (L[f]\otimes \nabla_{\bx} f)=\sigma \Delta_{\bv}\nabla_{\bx} f-\nabla_{\bx}L[f] \cdot \nabla_{\bv} f- f\nabla_{\bx} \nabla_{\bv} \cdot L[f].
\end{equation}
Multiplying \eqref{eq-ns-dx} by $2(1+\bv^2)^2 \nabla_{\bx} f$, we have
\begin{equation} \label{eq-ns-wdx}
 \begin{aligned}
  &(|\nabla_{\bx} f|^2(1+\bv^2)^2)_t+\bv \cdot \nabla_{\bx} (|\nabla_{\bx} f|^2(1+\bv^2)^2)+ \nabla_{\bv} \cdot (L[f]|\nabla_{\bx} f|^2(1+\bv^2)^2)\\
  =&2\sigma (1+\bv^2)^2\Delta_{\bv}\nabla_{\bx} f \cdot \nabla_{\bx} f-(1+\bv^2)^2|\nabla_{\bx} f|^2\nabla_{\bv} \cdot L[f]\\
  &+|\nabla_{\bx} f|^2 L[f]\cdot \nabla_{\bv}(1+\bv^2)^2-2(1+\bv^2)^2\nabla_{\bx} f \cdot(\nabla_{\bx} L[f]\cdot \nabla_{\bv}f+f\nabla_{\bx}\nabla_{\bv} \cdot L[f]).
 \end{aligned}
\end{equation}
Integrating $\eqref{eq-ns-wdx}$ over $\bbr^{2d}$, we deduce that
\begin{equation} \label{eq-ns-wdxit}
 \begin{aligned}
   \frac{d}{dt}\|(1+\bv^2)^{\frac12}\nabla_{\bx} f\|_{L^2(\nu)}^2\le&2\sigma \int_{\bbr^{2d}} (1+\bv^2)^2\Delta_{\bv}\nabla_{\bx} f \cdot \nabla_{\bx} f d\bx d\bv\\
   &- \int_{\bbr^{2d}}(1+\bv^2)^2|\nabla_{\bx} f|^2\nabla_{\bv} \cdot L[f] d\bx d\bv\\
   &+\int_{\bbr^{2d}}|\nabla_{\bx} f|^2 L[f]\cdot \nabla_{\bv}(1+\bv^2)^2 d\bx d\bv\\
   &-\int_{\bbr^{2d}}2(1+\bv^2)^2\nabla_{\bx} f \cdot(\nabla_{\bx}L[f]\cdot \nabla_{\bv}f+f\nabla_{\bx}\nabla_{\bv} \cdot L[f]) d\bx d\bv.
 \end{aligned}
\end{equation}
We estimate the right-hand side of \eqref{eq-ns-wdxit} term by term.
\begin{equation*}
 \begin{aligned}
 &2\sigma \int_{\bbr^{2d}} (1+\bv^2)^2\Delta_{\bv}\nabla_{\bx} f \cdot \nabla_{\bx} f d\bx d\bv \\
 =& -2\sigma \|(1+\bv^2)^{\frac12}\nabla_{\bv} \nabla_{\bx} f\|_{L^2(\nu)}^2 +\sigma \int_{\bbr^{2d}}|\nabla_{\bx} f|^2 \Delta_{\bv}(1+\bv^2)^2  d\bx d\bv,
 \end{aligned}
\end{equation*}
\begin{equation*}
  \begin{aligned}
 &- \int_{\bbr^{2d}}(1+\bv^2)^2|\nabla_{\bx} f|^2\nabla_{\bv} \cdot L[f] d\bx d\bv \\
 \le &C\|f\|_{L^1(\bbr^{2d})}\|(1+\bv^2)^{\frac12}\nabla_{\bx} f\|_{L^2(\nu)}^2,
  \end{aligned}
\end{equation*}
\begin{equation*}
 \begin{aligned}
 &\int_{\bbr^{2d}}|\nabla_{\bx} f|^2 L[f]\cdot \nabla_{\bv}(1+\bv^2)^2 d\bx d\bv \\
 \le& C\|(1+\bv^2)^{\frac12}f\|_{L^1(\bbr^{2d})}\|(1+\bv^2)^{\frac12}\nabla_{\bx} f\|_{L^2(\nu)}^2,
 \end{aligned}
\end{equation*}
\begin{equation*}
 \begin{aligned}
  &-\int_{\bbr^{2d}}2(1+\bv^2)^2\nabla_{\bx} f \cdot \Big(\nabla_{\bx} L[f]\cdot \nabla_{\bv}f+f\nabla_{\bx}\nabla_{\bv} \cdot L[f] \Big) d\bx d\bv\\
  =&\int_{\bbr^{2d}} 2f \nabla_{\bx} L[f]:\left[(1+\bv^2)^2 \nabla_{\bx}\nabla_{\bv} f+\nabla_{\bx}f \otimes \nabla_{\bv}(1+\bv^2)^2 \right]d\bx d\bv\\
  \le & \|(1+\bv^2)^{\frac12}f\|_{L^1(\bbr^{2d})}\|(1+\bv^2)^{\frac12} f\|_{L^2(\omega)}\|(1+\bv^2)^{\frac12}\nabla_{\bv} \nabla_{\bx} f\|_{L^2(\nu)}\\
  &+\|(1+\bv^2)^{\frac12}f\|_{L^1(\bbr^{2d})}\|(1+\bv^2)^{\frac12}f\|_{L^2(\omega)}\|(1+\bv^2)^{\frac12} \nabla_{\bx} f\|_{L^2(\nu)}\\
  \le & \sigma \|(1+\bv^2)^{\frac12}\nabla_{\bv} \nabla_{\bx} f\|_{L^2(\nu)}^2+C\|(1+\bv^2)^{\frac12}f\|_{L^1(\bbr^{2d})}^2\|(1+\bv^2)^{\frac12}f\|_{L^2(\omega)}^2\\
  &+\|(1+\bv^2)^{\frac12} \nabla_{\bx} f\|_{L^2(\nu)}^2
 \end{aligned}
\end{equation*}
Substituting these estimates into \eqref{eq-ns-wdxit}, we obtain
\begin{equation} \label{eq-ns-wdxitst}
 \begin{aligned}
   &\frac{d}{dt}\|(1+\bv^2)^{\frac12}\nabla_{\bx} f\|_{L^2(\nu)}^2+\sigma \|(1+\bv^2)^{\frac12}\nabla_{\bv} \nabla_{\bx} f\|_{L^2(\nu)}^2\\
   \le&C \Big(1+\|(1+\bv^2)^{\frac12}f\|_{L^1(\bbr^{2d})}\Big)\|(1+\bv^2)^{\frac12} \nabla_{\bx} f\|_{L^2(\nu)}^2\\&+C\|(1+\bv^2)^{\frac12}f\|_{L^1(\bbr^{2d})}^2\|(1+\bv^2)^{\frac12}f\|_{L^2(\omega)}^2.
 \end{aligned}
\end{equation}
Applying $\nabla_{\bv}$ to \eqref{eq-kine-nois}, we deduce
\begin{equation} \label{eq-ns-dv}
 (\nabla_{\bv} f)_t +\bv \cdot \nabla_{\bx} (\nabla_{\bv} f)+ \nabla_{\bv} \cdot (L[f]\otimes \nabla_{\bv} f)=\sigma \Delta_{\bv}\nabla_{\bv} f-\nabla_{\bv}L[f] \cdot \nabla_{\bv} f- \nabla_{\bx} f.
\end{equation}
Multiplying \eqref{eq-kine-nois} by $2(1+\bv^2)\nabla_{\bv} f$, we get
\begin{equation} \label{eq-ns-wdv}
 \begin{aligned}
  &(|\nabla_{\bv} f|^2(1+\bv^2))_t+\bv \cdot \nabla_{\bx} (|\nabla_{\bv} f|^2(1+\bv^2))+ \nabla_{\bv} \cdot (L[f]|\nabla_{\bv} f|^2(1+\bv^2))\\
  =&2\sigma (1+\bv^2)\Delta_{\bv}\nabla_{\bv} f \cdot \nabla_{\bv} f+|\nabla_{\bv} f|^2 L[f]\cdot \nabla_{\bv}(1+\bv^2)-2(1+\bv^2)\nabla_{\bv} f \cdot \nabla_{\bx} f\\
  &-(1+\bv^2)|\nabla_{\bv} f|^2\nabla_{\bv} \cdot L[f]
  -2(1+\bv^2)\nabla_{\bv} f \cdot \nabla_{\bv} L[f]\cdot \nabla_{\bv}f .
 \end{aligned}
\end{equation}
Similarly, we have
\begin{equation} \label{eq-ns-wdvitst}
 \begin{aligned}
   &\frac{d}{dt}\|(1+\bv^2)^{\frac12}\nabla_{\bv} f\|_{L^2(\bbr^{2d})}^2+\sigma \|(1+\bv^2)^{\frac12}\nabla_{\bv} \nabla_{\bv} f\|_{L^2(\bbr^{2d})}^2\\
   \le&C \Big(1+\|(1+\bv^2)^{\frac12}f\|_{L^1(\bbr^{2d})}\Big)\|(1+\bv^2)^{\frac12} \nabla_{\bv} f\|_{L^2(\bbr^{2d})}^2\\&+\|(1+\bv^2)^{\frac12}\nabla_{\bx} f\|_{L^2(\nu)}^2.
 \end{aligned}
\end{equation}
Combining \eqref{eq-l2fwtint}, \eqref{eq-ns-wdxitst}, \eqref{eq-ns-wdvitst} and Lemma \ref{lm-prlm-nois}, we arrive at
\begin{equation} \label{eq-ns-xdt}
 \frac{d}{dt}\|(1+\bv^2)^{\frac12} f\|_{X}^2+\sigma \|(1+\bv^2)^{\frac12}\nabla_{\bv} f\|_{X}^2
   \le c(t)\|(1+\bv^2)^{\frac12} f\|_{X}^2.
\end{equation}
Solving the above Gronwall's inequality yields
\begin{equation} \label{eq-ns-fxnm}
 \begin{aligned}
 &\|(1+\bv^2)^{\frac12}f(t)\|_{X}^2+\sigma \int_0^t \|(1+\bv^2)^{\frac12}\nabla_{\bv}f(\tau)\|_{X}^2 d \tau \\
      \le& \|(1+\bv^2)^{\frac12}f_0\|_{X}^2 \exp\left(\int_0^t c(\tau)d \tau\right), \quad \forall t \ge 0.
 \end{aligned}
\end{equation}
(2) For two smooth solutions $f(t,\bx,\bv)$ and $g(t,\bx,\bv)$ with initial data $f_0$ and $g_0$ satisfying the above initial conditions, respectively. We define
\[
 \bar{h}:=f-g, \quad \overline{F}:=\nabla_{\bx}f-\nabla_{\bx}g \quad \text{and} \quad \overline{G}:=\nabla_{\bv}f-\nabla_{\bv}g.
\]
It follows from \eqref{eq-ns-dx} that
\begin{equation} \label{eq-ns-dfr-dx}
 \begin{aligned}
   &\overline{F}_t + \bv \cdot \nabla_{\bx} \overline{F}+ \nabla_{\bv} \cdot (L[f]\otimes \overline{F})\\
   =&\sigma \Delta_{\bv}\overline{F}-\nabla_{\bv} \cdot (L[\bar{h}]\otimes \nabla_{\bx}g)-\nabla_{\bx}L[\bar{h}]\cdot \nabla_{\bv}f\\
   &-\nabla_{\bx}L[g]\cdot \nabla_{\bv}\bar{h}- \bar{h}\nabla_{\bx}\nabla_{\bv} \cdot L[g]-f \nabla_{\bx}\nabla_{\bv} \cdot L[\bar{h}].
 \end{aligned}
\end{equation}
Multiplying \eqref{eq-ns-dfr-dx} by $2(1+\bv^2)\overline{F}$, we obtain
\begin{equation} \label{eq-ns-dfr-dxe}
 \begin{aligned}
  &((1+\bv^2)\overline{F}^2)_t+\bv \cdot \nabla_{\bx} ((1+\bv^2)\overline{F}^2)+ \nabla_{\bv} \cdot (L[f] (1+\bv^2)\overline{F}^2)\\
  =&2\sigma (1+\bv^2)\Delta_{\bv} \overline{F} \cdot \overline{F}-2(1+\bv^2)\nabla_{\bv} \cdot (L[\bar{h}]\otimes \nabla_{\bx}g) \cdot \overline{F}\\
  &+\overline{F}^2 L[f]\cdot \nabla_{\bv}(1+\bv^2)-(1+\bv^2) \overline{F}^2\nabla_{\bv} \cdot L[f]\\
  &-2(1+\bv^2)\overline{F} \cdot \Big(\nabla_{\bx}L[\bar{h}]\cdot \nabla_{\bv}f+f \nabla_{\bx}\nabla_{\bv} \cdot L[\bar{h}]+\nabla_{\bx}L[g]\cdot \nabla_{\bv}\bar{h}+ \bar{h}\nabla_{\bx}\nabla_{\bv} \cdot L[g]\Big).
 \end{aligned}
\end{equation}
Integrating \eqref{eq-ns-dfr-dxe} over $\bbr^{2d}$, we have
\begin{equation} \label{eq-ns-dfr-dxit}
 \begin{aligned}
   \frac{d}{dt}\|\overline{F}\|_{L^2(\nu)}^2=&2\sigma \int_{\bbr^{2d}}(1+\bv^2)\Delta_{\bv} \overline{F} \cdot \overline{F}d \bx d\bv\\
   &-\int_{\bbr^{2d}}2(1+\bv^2)\nabla_{\bv} \cdot (L[\bar{h}]\otimes \nabla_{\bx}g) \cdot \overline{F}d \bx d\bv\\
   &+\int_{\bbr^{2d}}\Big(\overline{F}^2 L[f]\cdot \nabla_{\bv}(1+\bv^2)-(1+\bv^2) \overline{F}^2\nabla_{\bv} \cdot L[f]\Big) d \bx d\bv\\
   &-2\int_{\bbr^{2d}}(1+\bv^2)\overline{F} \cdot \Big(\nabla_{\bx}L[\bar{h}]\cdot \nabla_{\bv}f+f \nabla_{\bx}\nabla_{\bv} \cdot L[\bar{h}]\\
   &\qquad \quad+\nabla_{\bx}L[g]\cdot \nabla_{\bv}\bar{h}+ \bar{h}\nabla_{\bx}\nabla_{\bv} \cdot L[g]\Big)d \bx d\bv\\
  =&\sum_{i=1}^{4} J_i
 \end{aligned}
\end{equation}
We estimate each $J_i$ as follows.
\[
J_1=-2\sigma \|\nabla_{\bv}\overline{F}\|_{L^2(\nu)}^2+\sigma \int_{\bbr^{2d}}\overline{F}^2\Delta_{\bv}(1+\bv^2)d \bx d\bv,
\]
\begin{equation*}
 \begin{aligned}
  J_2=&2\int_{\bbr^{2d}}(L[\bar{h}]\otimes \nabla_{\bx}g): \nabla_{\bv}((1+\bv^2)\overline{F})d \bx
  d\bv\\
  \le&C\|(1+\bv^2)^{\frac12}\nabla_{\bx}g\|_{L^2(\nu)}\|\nabla_{\bv}\overline{F}\|_{L^2(\nu)}\|\bar{h}\|
  _{L^2(\omega)}+C\|\nabla_{\bx}g\|_{L^2(\nu)}\|\bar{h}\|_{L^2(\omega)}
  \|\overline{F}\|_{L^2(\nu)}\\
  \le&\sigma \|\nabla_{\bv}\overline{F}\|_{L^2(\nu)}^2+C\|(1+\bv^2)^{\frac12}\nabla_{\bx}g\|_{L^2(\nu)}^2\|\bar{h}\|
  _{L^2(\omega)}^2+\|\overline{F}\|_{L^2(\nu)}^2,
 \end{aligned}
\end{equation*}
\begin{equation*}
  J_3 \le C\|(1+\bv^2)^{\frac12}f\|_{L^1(\bbr^{2d})}\|\overline{F}\|_{L^2(\nu)}^2
  +C\|f\|_{L^1(\bbr^{2d})}\|\overline{F}\|_{L^2(\nu)}^2,
\end{equation*}
\begin{equation*}
 \begin{aligned}
 J_4=&2\int_{\bbr^{2d}} \nabla_{\bx}\nabla_{\bv}(\bar{h}(1+\bv^2)):\Big( \bar{h}\nabla_{\bx}L[g]+f\nabla_{\bx}L[\bar{h}]\Big)d\bx d\bv\\
 =&\int_{\bbr^{2d}}\Big(2(1+\bv^2) \nabla_{\bx}\overline{G}+4\overline{F}\otimes \bv\Big):\Big( \bar{h}\nabla_{\bx}L[g]+f\nabla_{\bx}L[\bar{h}]\Big)d\bx d\bv\\
 \le&C\|(1+\bv^2)^{\frac12}g\|_{L^1(\bbr^{2d})}\|\nabla_{\bv}\overline{F}\|_{L^2(\nu)}\|\bar{h}\|
  _{L^2(\omega)}+C\|(1+\bv^2)^{\frac12}g\|_{L^1(\bbr^{2d})}\|\bar{h}\|
  _{L^2(\omega)}\|\overline{F}\|_{L^2(\nu)}\\
  &+C\|(1+\bv^2)^{\frac12}f\|_{L^2(\omega)}\|\nabla_{\bv}\overline{F}\|_{L^2(\nu)}\|\bar{h}\|
  _{L^2(\omega)}+C\|f\|_{L^2(\omega)}\|\bar{h}\|
  _{L^2(\omega)}\|\overline{F}\|_{L^2(\nu)}\\
  \le&\sigma \|\nabla_{\bv}\overline{F}\|_{L^2(\nu)}^2+C\Big(\|(1+\bv^2)^{\frac12}g\|_{L^1(\bbr^{2d})}^2 +\|(1+\bv^2)^{\frac12}f\|_{L^2(\omega)}^2\Big)\|\bar{h}\|
  _{L^2(\omega)}^2+C\|\overline{F}\|_{L^2(\nu)}^2.
 \end{aligned}
\end{equation*}
Substituting these estimates into \eqref{eq-ns-dfr-dxit}, we obtain
\begin{equation} \label{eq-ns-dfr-dxits}
 \begin{aligned}
 \frac{d}{dt}\|\overline{F}\|_{L^2(\nu)}^2 \le&C\Big(1+\|(1+\bv^2)^{\frac12}f\|_{L^1(\bbr^{2d})}\Big)\|\overline{F}\|_{L^2(\nu)}^2\\
 &+C\Big(\|(1+\bv^2)^{\frac12}g\|_{L^1(\bbr^{2d})}^2+\|(1+\bv^2)^{\frac12}f\|_{L^2(\omega)}^2
 +\|(1+\bv^2)^{\frac12}\nabla_{\bx}g\|_{L^2(\nu)}^2 \Big)\|\bar{h}\|
  _{L^2(\omega)}^2
 \end{aligned}
\end{equation}
It follows from \eqref{eq-ns-dv} that
\begin{equation} \label{eq-ns-dfr-dv}
 \begin{aligned}
   \overline{G}_t &+ \bv \cdot \nabla_{\bx} \overline{G}+ \nabla_{\bv} \cdot (L[f]\otimes \overline{G})\\
   =&\sigma \Delta_{\bv}\overline{G}-\nabla_{\bv} \cdot (L[\bar{h}]\otimes \nabla_{\bv}g)
   -\nabla_{\bv}L[\bar{h}]\cdot \nabla_{\bv} f- \nabla_{\bv}L[g] \cdot \overline{G}-\overline{F}.
 \end{aligned}
\end{equation}
Multiplying \eqref{eq-ns-dfr-dv} by $2\overline{G}$ yields
\begin{equation} \label{eq-ns-dfr-dve}
 \begin{aligned}
  &(\overline{G}^2)_t+\bv \cdot \nabla_{\bx} (\overline{G}^2)+ \nabla_{\bv} \cdot (L[f] \overline{G}^2)\\
  =&-\overline{G}^2\nabla_{\bv} \cdot L[f]-2\sigma \Delta_{\bv} \overline{G} \cdot \overline{G}-2\nabla_{\bv} \cdot (L[\bar{h}]\otimes \nabla_{\bv}g) \cdot \overline{G}\\
  &-2\overline{G} \cdot \nabla_{\bv}L[\bar{h}]\cdot \nabla_{\bv}f-2\overline{G}\cdot \nabla_{\bv}L[g]\cdot \overline{G}-2\overline{F} \cdot \overline{G}.
 \end{aligned}
\end{equation}
Integrating \eqref{eq-ns-dfr-dve} over $\bbr^{2d}$ and performing integration by parts, we have
\begin{equation} \label{eq-ns-dfr-dvit}
 \begin{aligned}
  \frac{d}{dt} \|\overline{G}\|_{L^2(\bbr^{2d})}\le & C(1+\|f\|_{L^1(\bbr^{2d})}+\|g\|_{L^1(\bbr^{2d})})\|\overline{G}\|_{L^2(\bbr^{2d})}^2
  + \|\overline{F}\|_{L^2(\nu)}^2\\
  &+C\Big(\|\nabla_{\bv}f\|_{L^2(\bbr^{2d})}^2 +\|(1+\bv^2)^{\frac12}\nabla_{\bv}g\|_{L^2(\bbr^{2d})}^2 \Big)\|\bar{h}\|_{L^2(\omega)}^2.
 \end{aligned}
\end{equation}
Combining \eqref{eq-dfrwtinmt}, \eqref{eq-ns-dfr-dxits}, \eqref{eq-ns-dfr-dvit} and using\eqref{eq-ns-fxnm}, Lemma \ref{lm-prlm-nois}, we have
\begin{equation} \label{eq-str-dfr}
 \frac{d}{dt}\|f-g\|_X^2\le c(t)\|f-g\|_X^2,
\end{equation}
which implies
\begin{equation} \label{eq-str-sta}
 \sup_{0\le t \le T}\|f(t)-g(t)\|_{X}\le \|f_0-g_0\|_{X} \exp\left(\int_0^T c(t)dt\right), \quad \forall T\ge 0.
\end{equation}
This completes the proof.
\end{proof}
\subsection{Proof of Theorem \ref{tm-wek-nois} and \ref{tm-str-nois}}
We first mollify the initial data by convolution,i.e.,$$f_0^{\e}(\bx, \bv)=f_0 \ast j_{\e}(\bx, \bv),$$ where $j_{\e}$ is the standard mollifier. Using the contraction principle, we can obtain the local smooth solution by the standard procedure. Combining with the a priori estimate in Lemma \ref{lm-prlm-nois}(3), one can extend the local smooth solution to be global-in-time.

Then using the stability estimate in Lemma \ref{lm-stb-nois}, we infer that
\begin{equation} \label{eq-wkcov-ns}
\sup_{0\le t \le T}\|f^{\e_i}(t)-f^{\e_j}(t)\|_{L^{2}(\omega)}\le \|f_0^{\e_i}-f_0^{\e_j}\|_{L^{2}(\omega)}\exp\left(\int_0^T c(t)dt\right), \quad \forall T\ge 0,
\end{equation}
where $f^{\e_i}(t,\bx, \bv)$ and $f^{\e_j}(t,\bx, \bv)$ are smooth solution with initial data $f_0^{\e_i}$ and $f_0^{\e_j}$, respectively.
From \eqref{eq-wkcov-ns}, we know there exists $f(t,\bx, \bv) \in C([0,T],L^2(\omega))$ such that
\[
 f^{\e_i}(t,\bx, \bv) \to f(t,\bx, \bv) \quad \text{in $C([0,T],L^2(\omega))\ \forall T>0$, as $\e_i \to 0$}.
\]
Due to the arbitrariness of $T$, we know $f(t,\bx, \bv) \in C([0,+\infty),L^2(\omega))$. It is easy to see that $f(t,\bx, \bv)$ is a weak solution to \eqref{eq-kine-nois}.
Take smooth initial data $f_0^{\e_i}$ and $g_0^{\e_i}$. We also have
\begin{equation} \label{eq-wkstacov}
\sup_{0\le t \le T}\|f^{\e_i}(t)-g^{\e_i}(t)\|_{L^{2}(\omega)}\le \|f_0^{\e_i}-g_0^{\e_i}\|_{L^{2}(\omega)}\exp\left(\int_0^T c(t)dt\right), \quad \forall T\ge 0,
\end{equation}
Letting $\e_i \to 0$, we obtain the stability estimate for weak solutions to \eqref{eq-kine-nois}, which also amounts to uniqueness of the weak solution. While Theorem \ref{tm-wek-nois}(1) can be easily proved by using the lower semi-continuity of the weakly and weakly-$\star$ convergent sequence.

Theorem \ref{tm-str-nois} can be proved in the same way. We omit its proof for brevity. Thus we complete the proof.
%
%
\section{Vanishing Noise Limit}
\setcounter{equation}{0}
In this section, we study the vanishing noise limit as $\sigma$ tends to $0$. In fact, we can pass to the limit for both weak and strong solutions to \eqref{eq-kine-nois}. Our results are as follows.
\begin{theorem} \label{tm-wek-vnois}
Assume $(1+\bv^2)^{\frac12}f_0(\bx, \bv)\in  L^2(\omega)$. Then \eqref{eq-kine-wtnois} admits a unique weak solution $f(t, \bx,\bv) \in C([0,+\infty),L^2(\omega))$.
Moreover, it hold that
\begin{eqnarray*}
    &&(1)~\|(1+\bv^2)^{\frac12}f(t)\|_{L^2(\omega)} \le \|(1+\bv^2)^{\frac12}f_0\|_{L^2(\omega)} \exp\left(\int_0^t c(\tau)d \tau\right), \quad a.e.\ t \ge 0;\\
    &&(2)~\sup_{0\le t \le T}\|f(t)-g(t)\|_{L^2(\omega)}\le \|f_0-g_0\|_{L^2(\omega)}\exp\left(\int_0^T c(t)dt\right), \quad \forall T\ge 0,
  \end{eqnarray*}
where $f(t, \bx,\bv)$ and $g(t, \bx,\bv)$ are weak solutions with initial data $f_0$ and $g_0$ satisfying the above conditions, respectively.
\end{theorem}

\begin{theorem} \label{tm-str-vnois}
Assume $(1+\bv^2)^{\frac12}f_0(\bx, \bv)\in  X$. Then \eqref{eq-kine-wtnois} admits a unique strong solution $f(t,\bx, \bv)\in C([0,+\infty), X)$.
Moreover, it hold that
\begin{eqnarray*}
    &&(1)~\|(1+\bv^2)^{\frac12}f(t)\|_{X} \le \|(1+\bv^2)^{\frac12}f_0\|_{X} \exp\left(\int_0^t c(\tau)d \tau\right), \quad a.e.\ t \ge 0;\\
    &&(2)~\sup_{0\le t \le T}\|f(t)-g(t)\|_{X}\le \|f_0-g_0\|_{X} \exp\left(\int_0^T c(t)dt\right), \quad \forall T\ge 0,
  \end{eqnarray*}
where $f(t, \bx,\bv)$ and $g(t, \bx,\bv)$ are strong solutions with initial data $f_0$ and $g_0$ satisfying the above conditions, respectively.
\end{theorem}

In fact, Remark \ref{rm-cla} still holds as $\sigma \to 0$. The following velocity averaging lemma is due to \cite{diperna1989global}(Theorem 5 and Remark 3 of Theorem 3). We mainly use it to get some compactness from the weak solution sequence  to \eqref{eq-kine-nois}.
\begin{lemma}[DiPerna and Lions] \label{lm-vel}
Let $m \ge 0$, $f, g \in L^2(\bbr \times \bbr^{2d})$ and $f(t,\bx,\bv), g(t,\bx,\bv)$ satisfy
\[
  \frac{\partial f}{\partial t}+\bv \cdot \nabla_{\bx} f=\nabla_{\bv}^{\xi}g \quad \text{in} \ \mathcal{D}',
\]
where $\nabla_{\bv}^{\xi}=\partial_{\bv^1}^{\xi^1}\partial_{\bv^2}^{\xi^2}\cdots \partial_{\bv^d}^{\xi^d}$ and $|\xi|= \sum_{i=1}^{d} \xi^i=m$. Then for any $\phi(\bv) \in C_c^{\infty}(\bbr^{d})$, it holds that
\[
 \left\|\int_{\bbr^{d}}f(t,\bx,\bv) \phi(\bv)d \bv\right\|_{H^s(\bbr \times \bbr^{d})}\le C \left(\|f\|_{L^2(\bbr \times \bbr^{2d})}+\|g\|_{L^2(\bbr \times \bbr^{2d})} \right),
\]
where $s=\frac{1}{2(1+m)}$ and $C$ is a positive constant.
\end{lemma}

Denote the solution to \eqref{eq-kine-nois} by $f^{\sigma}(t,\bx, \bv)$. Then we present the proof of the above two theorems.
\vskip 0.3cm
\noindent \textit{Proof of Theorem \ref{tm-wek-vnois}.} According to Theorem \ref{tm-wek-nois}, we know
\begin{equation} \label{eq-vwel2estn}
 \begin{aligned}
  &\|(1+\bv^2)^{\frac12}f^{\sigma}(t)\|_{L^2(\omega)}^2+\sigma \int_0^t \|(1+\bv^2)^{\frac12}\nabla_{\bv}f^{\sigma}(\tau)\|_{L^2(\omega)}^2 d \tau \\
      \le& \|(1+\bv^2)^{\frac12}f_0\|_{L^2(\omega)}^2 \exp\left(\int_0^t c(\tau)d \tau\right), \quad a.e.\ t \ge 0.
 \end{aligned}
\end{equation}
Thus there exists a sequence $\{f^{\sigma_j}(t, \bx, \bv)\}$ such that
\begin{equation} \label{eq-wcov}
 f^{\sigma_j}(t, \bx, \bv) \rightharpoonup f(t, \bx, \bv) \quad  \text{weakly-$\star$ in $L^{\infty}((0,T), L^2(\bbr^{2d}))\ \forall T>0$, as $\sigma_j \to 0$}.
\end{equation}
This also implies that
\begin{equation} \label{eq-wl2d}
 \int_0^T\int_{\bbr^{2d}}(1+\bv^2)^{\frac12}\omega^{\frac12}f^{\sigma_j}\psi d\bx d\bv dt \to \int_0^T\int_{\bbr^{2d}}(1+\bv^2)^{\frac12}\omega^{\frac12}f\psi d\bx d\bv dt,\quad \text{as $\sigma_j \to 0$},
\end{equation}
for any $\psi(t, \bx, \bv) \in C_c^{\infty}((0,T)\times \bbr^{2d})$.
Thus we have
\begin{equation} \label{eq-vwfd}
 (1+\bv^2)^{\frac12}\omega^{\frac12}f^{\sigma_j} \to (1+\bv^2)^{\frac12}\omega^{\frac12}f \quad \text{in $\mathcal{D}'$, as $\sigma_j \to 0$}.
\end{equation}
Combining with \eqref{eq-vwel2estn} and using the uniqueness of limit in the sense of distributions,
we also have
\begin{equation} \label{eq-vwfl}
 (1+\bv^2)^{\frac12}\omega^{\frac12}f^{\sigma_j} \rightharpoonup (1+\bv^2)^{\frac12}\omega^{\frac12}f,
\end{equation}
weakly-$\star$ in $L^{\infty}((0,T), L^2(\bbr^{2d}))$ $\forall T>0$, as $\sigma_j \to 0$.
Therefore, we have
\begin{equation} \label{eq-vwfles}
 \textrm{ess}\sup_{0\le t \le T} \|(1+\bv^2)^{\frac12}f(t)\|_{L^2(\omega)} \le \|(1+\bv^2)^{\frac12}f_0\|_{L^2(\omega)} \exp\left(\int_0^T c(t)d t\right),  \quad \forall T>0.
\end{equation}
Next we prove
\[
 L[f^{\sigma_j}]f^{\sigma_j} \to L[f]f \quad \text{in $\mathcal{D}'((0,T)\times \bbr^{2d})$, as $\sigma_j \to 0$}.
\]
We write \eqref{eq-kine-nois} in the form of
\begin{equation} \label{eq-kine-noisrefm}
     \frac{\partial f^{\sigma_j}}{\partial t}  + \bv \cdot \nabla_{\bx} f^{\sigma_j}=\nabla_{\bv} \cdot \Big(\sigma_j \nabla_{\bv} f^{\sigma_j}-L[f^{\sigma_j}]f^{\sigma_j}\Big).
\end{equation}
Using Lemma \ref{lm-vel} and combining with \eqref{eq-vwel2estn}, we have
\[
 \left\|\int_{\bbr^{d}}f(t,\bx,\bv) \phi(\bv)d \bv\right\|_{H^{\frac14}([0,T] \times \bbr^{d})}\le C, \quad \forall \phi(\bv) \in C_c^{\infty}(\bbr^{d}).
\]
Since
\[
 H^{\frac14}\left([0,T]\times K \right) \hookrightarrow \hookrightarrow L^1\left([0,T]\times K\right), \quad \text{$\forall K$ compact in $\bbr^{d}$},
\]
there exists a subsequence, still denoted by $\{f^{\sigma_j}\}$, such that
\begin{equation} \label{eq-vloc}
 \int_{\bbr^{d}} f^{\sigma_j}\phi(\bv)d\bv \to \int_{\bbr^{d}} f \phi(\bv)d\bv \quad \text{in $L_{loc}^1((0,T)\times \bbr^{d})$, as $\sigma_j \to 0$}.
\end{equation}
For any $\e >0$, if we choose $R$ suitably large, it hold that
\begin{equation} \label{eq-vtrv}
 \begin{aligned}
  &\int_0^T \int_{\bbr^{d}}\left|\int_{|\bv|>R}(1+\bv^2)^{\frac12}f^{\sigma_j}d\bv \right|d\bx dt\\
  \le&\left(\int_0^T \int_{\bbr^{d}}\int_{|\bv|>R}\omega^{-1} d\bv d\bx dt\right)^{\frac12}\left(\int_0^T \|(1+\bv^2)^{\frac12}f^{\sigma_j}(t)\|_{L^2(\omega)}^2 dt\right)^{\frac12}\\
  \le&C(T)\e
 \end{aligned}
\end{equation}
and
\begin{equation} \label{eq-vtrx}
 \begin{aligned}
  &\int_0^T \int_{|\bx|>R}\left|\int_{\bbr^{d}}(1+\bv^2)^{\frac12}f^{\sigma_j}d\bv \right|d\bx dt\\
  \le&\left(\int_0^T \int_{|\bx|>R}\int_{\bbr^{d}}\omega^{-1} d\bv d\bx dt\right)^{\frac12}\left(\int_0^T \|(1+\bv^2)^{\frac12}f^{\sigma_j}(t)\|_{L^2(\omega)}^2 dt\right)^{\frac12}\\
  \le&C(T)\e.
 \end{aligned}
\end{equation}
Combining \eqref{eq-vloc}, \eqref{eq-vtrv} and \eqref{eq-vtrx}, we have
\begin{equation} \label{eq-vglo}
 \int_{\bbr^{d}} (1+\bv^2)^{\frac12}f^{\sigma_j}d\bv \to \int_{\bbr^{d}} (1+\bv^2)^{\frac12}f d\bv \quad \text{in $L^1((0,T)\times \bbr^{d})$, as $\sigma_j \to 0$}.
\end{equation}
It follows from \eqref{eq-vglo} that there exists a subsequence, still denoted by $\{f^{\sigma_j}\}$, such that
\begin{equation} \label{eq-vlae}
 L[f^{\sigma_j}] \to L[f], \quad\text{a.e. \ $(t, \bx, \bv)\in [0,T]\times \bbr^{2d}$, as $\sigma_j \to 0$}.
\end{equation}
Due to \eqref{eq-vwel2estn}, it is easy to see
\begin{equation} \label{eq-vlbd}
 \left| \int_{\bbr^{2d}} \varphi(|\bx-\by|)f^{\sigma_j}(t, \by,\bv^*)(1+|\bv^*|^2)^{\frac12}d \by d\bv^*\right|\le C(T), \quad \forall t \in [0,T].
\end{equation}
Combining \eqref{eq-vloc}, \eqref{eq-vlae} and \eqref{eq-vlbd}, we deduce that
\begin{equation} \label{eq-vdist}
 \int_0^T \int_{\bbr^{2d}}L[f^{\sigma_j}]f^{\sigma_j}\phi_1(\bv)\phi_2(t,\bx)d\bv d\bx dt \to \int_0^T \int_{\bbr^{2d}}L[f]f\phi_1(\bv)\phi_2(t,\bx)d\bv d\bx dt,
\end{equation}
for any $\phi_1(\bv) \in C_c^{\infty}(\bbr^{d})$, $\phi_2(t, \bx) \in C_c^{\infty}((0,T)\times \bbr^{d})$, as $\sigma_j \to 0$. Using the density of the sums of the function with the form $\phi_1(\bv)\phi_2(t,\bx)$ in $C_c^{\infty}((0,T)\times \bbr^{2d})$, we have
\[
 L[f^{\sigma_j}]f^{\sigma_j} \to L[f]f \quad \text{in $\mathcal{D}'((0,T)\times \bbr^{2d})$, as $\sigma_j \to 0$}.
\]
Therefore,
\begin{equation} \label{eq-vwek}
 \frac{\partial f}{\partial t}  + \bv \cdot \nabla_{\bx} f +\nabla_{\bv} \cdot (L[f]f)=0 \quad \text{in $\mathcal{D}'((0,T)\times \bbr^{2d})$}.
\end{equation}
 Then we prove $f \in C([0,T], L^2(\omega))$. Since $(1+\bv^2)^{\frac12}f \in L^{\infty}((0,T), L^2(\omega))$, by the interpolation, we only need to prove $f \in C([0,T], L^2(\bbr^{2d}))$.
From \eqref{eq-vwek} and \eqref{eq-vwfles}, we know
\begin{equation} \label{eq-vft}
 f_t \in L^{\infty}((0,T), H^{-1}(\bbr^{2d}))
\end{equation}
Combining with the fact that $f \in L^{\infty}((0,T), L^{2}(\bbr^{2d}))$, we know
\begin{equation} \label{eq-vfwc}
 f \in C([0,T], L^{2}(\bbr^{2d})-W),
\end{equation}
which means $f$ is continuous in $[0,T]$ with respect to the weak topology in $L^{2}(\bbr^{2d})$.
In the following, we prove $\|f(t)\|_{L^{2}(\bbr^{2d})} \in C[0,T]$.

Take the standard mollifier $j_{\e}(\bx-\cdot, \bv-\cdot)$ as the test function in \eqref{eq-vwek}. Denoting $f*j_{\e}(\bx, \bv)$ by $\langle f \rangle_{\e}$, we have
\begin{equation} \label{eq-vtes}
 \begin{aligned}
 &\left(\langle f \rangle_{\e}\right)_t+\bv \cdot \nabla_{\bx}\langle f \rangle_{\e}+\nabla_{\bv} \cdot \left(L[f]\langle f \rangle_{\e}\right)\\
 =&-\nabla_{\bx}\cdot \langle \bv f \rangle_{\e}+\bv \cdot \nabla_{\bx}\langle f \rangle_{\e}-\nabla_{\bv} \cdot \langle L[f] f \rangle_{\e}+\nabla_{\bv} \cdot \left(L[f]\langle f \rangle_{\e}\right)
 \end{aligned}
\end{equation}
Multiplying \eqref{eq-vtes} by $2\langle f \rangle_{\e}$ yields
\begin{equation} \label{eq-vteg}
 \begin{aligned}
   &\frac{\partial}{\partial t}\langle f \rangle_{\e}^2 + \bv \cdot \nabla_{\bx} \langle f \rangle_{\e}^2 +\nabla_{\bv} \cdot \left(L[f]\langle f \rangle_{\e}^2\right)\\
   =&-\langle f \rangle_{\e}^2 \nabla_{\bv} \cdot L[f]-2\left[\nabla_{\bx}\cdot \langle \bv f \rangle_{\e}-\bv \cdot \nabla_{\bx}\langle f \rangle_{\e} \right]\cdot \langle f \rangle_{\e}\\
   &-2\left[ \nabla_{\bv} \cdot \langle L[f] f \rangle_{\e}-\nabla_{\bv} \cdot \left(L[f]\langle f \rangle_{\e}\right)\right]\cdot \langle f \rangle_{\e}
 \end{aligned}
\end{equation}
Integrating \eqref{eq-vteg} over $\bbr^{2d}$, we obtain
\begin{equation} \label{eq-vfe2int}
 \begin{aligned}
 \frac{d}{d t}\|\langle f \rangle_{\e}\|_{L^2(\bbr^{2d})}^2=&-\int_{\bbr^{2d}}\langle f \rangle_{\e}^2 \nabla_{\bv} \cdot L[f] d\bx d\bv\\
 &-\int_{\bbr^{2d}}2\left[\nabla_{\bx}\cdot \langle \bv f \rangle_{\e}-\bv \cdot \nabla_{\bx}\langle f \rangle_{\e} \right]\cdot \langle f \rangle_{\e} d\bx d\bv\\
 &-\int_{\bbr^{2d}}2\left[ \nabla_{\bv} \cdot \langle L[f] f \rangle_{\e}-\nabla_{\bv} \cdot \left(L[f]\langle f \rangle_{\e}\right)\right]\cdot \langle f \rangle_{\e} d\bx d\bv\\
 =&\sum_{i=1}^{3}K_i.
 \end{aligned}
\end{equation}
We estimate each $K_i(1\le i \le 3)$ as follows.
\[
 |K_1|\le C\|f\|_{L^1(\bbr^{2d})}\|f\|_{L^2(\bbr^{2d})}^2,
\]
\begin{equation*}
 \begin{aligned}
 |K_2|=&2\left|\int_{\bbr^{2d}}\int_{\bbr^{2d}}(\bw-\bv)\cdot \nabla_{\bx}j_{\e}(\bx-\bz, \bv-\bw)f(t, \bz, \bw)d\bz d\bw \langle f \rangle_{\e}d\bx d\bv\right|\\
 \le&2C\|f\|_{L^2(\bbr^{2d})}^2,
 \end{aligned}
\end{equation*}
where we have used the fact
\[
 |\bw-\bv|\le \e \quad \text{and}\quad \|\nabla_{\bx}j_{\e}\|_{L^1(\bbr^{2d})}\le \frac{C}{\e}.
\]
\begin{equation*}
 \begin{aligned}
 |K_3|=&2\left|\int_{\bbr^{2d}}\nabla_{\bv}\cdot \int_{\bbr^{2d}} \Big (L[f](\bz,\bw)-L[f](\bx,\bv)\Big)j_{\e}(\bx-\bz, \bv-\bw)f(t, \bz, \bw)d\bz d\bw \langle f \rangle_{\e}d\bx d\bv \right|\\
 \le&2C\|f\|_{L^1(\bbr^{2d})}\|f\|_{L^2(\bbr^{2d})}^2\\
 &+2\left|\int_{\bbr^{2d}} \int_{\bbr^{2d}} \Big (L[f](\bz,\bw)-L[f](\bx,\bv)\Big)\cdot\nabla_{\bv}j_{\e}(\bx-\bz, \bv-\bw)f(t, \bz, \bw)d\bz d\bw \langle f \rangle_{\e}d\bx d\bv \right|\\
 \le& 2C\|f\|_{L^1(\bbr^{2d})}\|f\|_{L^2(\bbr^{2d})}^2+2C\|(1+\bv^2)^{\frac12}f\|_{L^1(\bbr^{2d})}
 \|(1+\bv^2)^{\frac12}f\|_{L^2(\bbr^{2d})}\|f\|_{L^2(\bbr^{2d})}
 \end{aligned}
\end{equation*}
Substituting these estimates into \eqref{eq-vfe2int} and integrating over $[t_1, t_2]$, $\forall t_1, t_2 \in [0,T]$, we obtain
\[
 \left| \|\langle f(t_2) \rangle_{\e}\|_{L^2(\bbr^{2d})}^2-\|\langle f(t_1) \rangle_{\e}\|_{L^2(\bbr^{2d})}^2 \right|\le C(T)|t_2-t_1|,\quad \forall t_1, t_2 \in [0,T].
\]
where we have used \eqref{eq-vwfles}. Letting $\e \to 0$, we have
\begin{equation} \label{eq-vnctiu}
  \left| \| f(t_2)\|_{L^2(\bbr^{2d})}^2-\| f(t_1)\|_{L^2(\bbr^{2d})}^2 \right|\le C(T)|t_2-t_1|,\quad \forall t_1, t_2 \in [0,T].
\end{equation}
Combining \eqref{eq-vfwc} with \eqref{eq-vnctiu}, we have
\begin{equation} \label{eq-vfsc}
  f \in C([0,T], L^2(\bbr^{2d})),\quad T\ge 0.
\end{equation}
Due to the arbitrariness of $T$, from \eqref{eq-vwfles} and \eqref{eq-vfsc}, we have $f \in C([0,\infty), L^2(\omega))$ and
\[
 \|(1+\bv^2)^{\frac12}f(t)\|_{L^2(\omega)} \le \|(1+\bv^2)^{\frac12}f_0\|_{L^2(\omega)} \exp\left(\int_0^t c(\tau)d \tau\right), \quad a.e.\ t \ge 0
\]
Similarly, we can prove
\[
 \omega^{\frac12}f^{\sigma_j} \rightharpoonup \omega^{\frac12}f,
\quad \text{
weakly-$\star$ in $L^{\infty}((0,T), L^2(\bbr^{2d}))\ \forall T>0$, as $\sigma_j \to 0$}
\]
and
\[
 \omega^{\frac12}g^{\sigma_j} \rightharpoonup \omega^{\frac12}g,
\quad \text{
weakly-$\star$ in $L^{\infty}((0,T), L^2(\bbr^{2d}))\ \forall T>0$, as $\sigma_j \to 0$},
\]
for solution sequences $\{f^{\sigma_j}\}$ and $\{g^{\sigma_j}\}$ with initial data $f_0$ and $g_0$, respectively. Therefore, we have
\[
 \sup_{0\le t \le T}\|f(t)-g(t)\|_{L^2(\omega)}\le \|f_0-g_0\|_{L^2(\omega)}\exp\left(\int_0^T c(t)dt\right), \quad \forall T\ge 0,
\]
which implies uniqueness of the weak solution. This completes the proof.
\vskip 0.3cm
\noindent \textit{Proof of Theorem \ref{tm-str-vnois}}. According to Theorem \ref{tm-str-nois}, we have
\begin{equation} \label{eq-vsfw}
 \begin{aligned}
  &\|(1+\bv^2)^{\frac12}f^{\sigma}(t)\|_{X}^2+\sigma \int_0^t \|(1+\bv^2)^{\frac12}\nabla_{\bv}f^{\sigma}(\tau)\|_{X}^2 d \tau \\
     \le& \|(1+\bv^2)^{\frac12}f_0\|_{X}^2 \exp\left(\int_0^t c(\tau)d \tau\right), \quad a.e.\ t \ge 0.
 \end{aligned}
\end{equation}
From \eqref{eq-vsfw} and the equation \eqref{eq-kine-nois}, we know $f^{\sigma}$ are uniformly bounded in $L^{\infty}((0,T), L^2(\bbr^{2d}))$, and $\frac{\partial f^{\sigma}}{\partial t}$ are uniformly bounded in $L^{2}((0,T), L^2(\bbr^{2d}))$, $\forall T>0$.
It follows from the Ascoli-Arzela theorem that there exists a sequence $\{f^{\sigma_j}\}$ such that
\begin{equation} \label{eq-vscov}
 f^{\sigma_j}(t, \bx, \bv) \to f(t, \bx, \bv) \quad \text{in $C([0,T], L^2(\bbr^{2d}))$, as $\sigma_j \to 0$}.
\end{equation}
It is easy to see $f(t, \bx, \bv)$ is a weak solution to \eqref{eq-kine-wtnois}. Similarly as the proof in Theorem \ref{tm-wek-vnois}, we can show
\[
 (1+\bv^2)^{\frac12}f^{\sigma_j} \rightharpoonup (1+\bv^2)^{\frac12}f,\quad \text{weakly-$\star$ in $L^{\infty}((0,T), L^2(\omega))$},
\]
\[
 (1+\bv^2)^{\frac12}\nabla_{\bx} f^{\sigma_j} \rightharpoonup (1+\bv^2)^{\frac12}\nabla_{\bx}f,\quad \text{weakly-$\star$ in $L^{\infty}((0,T), L^2(\nu))$},
\]
and
\[
 (1+\bv^2)^{\frac12}\nabla_{\bv} f^{\sigma_j} \rightharpoonup (1+\bv^2)^{\frac12}\nabla_{\bv}f,\quad \text{weakly-$\star$ in $L^{\infty}((0,T), L^2(\bbr^{2d}))$}
\]
for any $T>0$, as $\sigma_j \to 0$. Thus we have
\begin{equation} \label{eq-vsfles}
 \textrm{ess}\sup_{0\le t \le T} \|(1+\bv^2)^{\frac12}f(t)\|_{X} \le \|(1+\bv^2)^{\frac12}f_0\|_{X} \exp\left(\int_0^T c(\tau)d \tau\right), \quad \forall T>0.
\end{equation}
Next we prove $f(t, \bx, \bv) \in C([0,T], X)\ \forall T>0$. In fact, using \eqref{eq-vsfles} and the interpolation, we only need to prove
\[
 f(t, \bx, \bv) \in C([0,T], W^{1,2}(\bbr^{2d})), \quad \forall T>0.
\]
Similarly as the proof in Theorem \ref{tm-wek-vnois}, we can easily show that
\begin{equation} \label{eq-vswctu}
 f(t, \bx, \bv) \in C([0,T], W^{1,2}(\bbr^{2d})-W), \quad \forall T>0,
\end{equation}
which means that $f$ is continuous in $[0,T]$ with respect to the weak topology in $W^{1,2}(\bbr^{2d})$. The following proof is devoted to demonstrating  that
\[
 \|f(t)\|_{W^{1,2}(\bbr^{2d})} \in C[0,T], \quad \forall T>0.
\]
Based on Theorem \ref{tm-wek-vnois}, we just prove that $\|\nabla_{\bx}f\|_{L^{2}(\bbr^{2d})}$ and $\|\nabla_{\bv}f\|_{L^{2}(\bbr^{2d})}$ are in $ C[0,T], \forall T>0.$
Since
\[
(\nabla_{\bx} f)_t +\bv \cdot \nabla_{\bx} (\nabla_{\bx} f)+ \nabla_{\bv} \cdot (L[f]\otimes \nabla_{\bx} f)=-\nabla_{\bx}L[f] \cdot \nabla_{\bv} f- f\nabla_{\bx} \nabla_{\bv} \cdot L[f]  \quad \text{in $ \mathcal{D}'((0,T) \times \bbr^{2d})$},
\]
taking the standard mollifier $j_{\e}(\bx-\cdot, \bv-\cdot)$ as the test function yields
\begin{equation} \label{eq-vstes}
 \begin{aligned}
  &(\langle \nabla_{\bx} f\rangle_{\e})_t +\bv \cdot \nabla_{\bx} \langle\nabla_{\bx} f\rangle_{\e}+ \nabla_{\bv} \cdot (L[f]\otimes \langle\nabla_{\bx} f\rangle_{\e})\\
  =&\langle-\nabla_{\bx}L[f] \cdot \nabla_{\bv} f- f\nabla_{\bx} \nabla_{\bv} \cdot L[f]\rangle_{\e}\\
  &-[\nabla_{\bx} \cdot \langle\bv \otimes \nabla_{\bx}f\rangle_{\e}-\bv \cdot \nabla_{\bx} \langle\nabla_{\bx} f\rangle_{\e}]\\
  &-[\nabla_{\bv}\cdot \langle L[f]\otimes \nabla_{\bx} f\rangle_{\e}-\nabla_{\bv} \cdot (L[f]\otimes \langle\nabla_{\bx} f\rangle_{\e})]
 \end{aligned}
\end{equation}
Multiplying \eqref{eq-vstes} by $2\langle \nabla_{\bx} f\rangle_{\e}$ and integrating the resulting equation over $\bbr^{2d}$, we obtain
\begin{equation} \label{eq-vstesint}
 \begin{aligned}
   \frac{d}{dt}\|\langle \nabla_{\bx} f\rangle_{\e}\|_{L^2(\bbr^{2d})}^{2}
   =&-\int_{\bbr^{2d}}\langle \nabla_{\bx} f\rangle_{\e}^2 \nabla_{\bv} \cdot L[f] d\bx d\bv\\
   &-2\int_{\bbr^{2d}}\langle \nabla_{\bx} f\rangle_{\e} \cdot\langle\nabla_{\bx}L[f] \cdot \nabla_{\bv} f+ f\nabla_{\bx} \nabla_{\bv} \cdot L[f]\rangle_{\e} d\bx d\bv\\
   &-2\int_{\bbr^{2d}}[\nabla_{\bx} \cdot \langle\bv \otimes \nabla_{\bx}f\rangle_{\e}-\bv \cdot \nabla_{\bx} \langle\nabla_{\bx} f\rangle_{\e}]\cdot \langle \nabla_{\bx} f\rangle_{\e} d\bx d\bv\\
   &-2\int_{\bbr^{2d}} [\nabla_{\bv}\cdot \langle L[f]\otimes \nabla_{\bx} f\rangle_{\e}-\nabla_{\bv} \cdot (L[f]\otimes \langle\nabla_{\bx} f\rangle_{\e})]\cdot \langle \nabla_{\bx} f\rangle_{\e} d\bx d\bv\\
   =&\sum_{i=1}^{4}H_i
 \end{aligned}
\end{equation}
We estimate each $H_i(1\le i \le 4)$ as follows.
\[
|H_1|\le \|\nabla_{\bv} \cdot L[f]\|_{L^{\infty}(\bbr^{2d})}\|\langle \nabla_{\bx} f\rangle_{\e}\|_{L^2(\bbr^{2d})}^{2}\le C\|f\|_{L^1(\bbr^{2d})}\|\nabla_{\bx} f\|_{L^2(\bbr^{2d})}^{2},
\]
\begin{equation*}
 \begin{aligned}
 |H_2|\le&\|\nabla_{\bx} f\|_{L^2(\bbr^{2d})}\|\nabla_{\bx}L[f] \cdot \nabla_{\bv} f+ f\nabla_{\bx} \nabla_{\bv} \cdot L[f]\|_{L^2(\bbr^{2d})}\\
 \le&C\|\nabla_{\bx} f\|_{L^2(\bbr^{2d})}\Big( \|(1+\bv^2)^{\frac12}f\|_{L^1(\bbr^{2d})}\|(1+\bv^2)^{\frac12}\nabla_{\bv}f\|_{L^2(\bbr^{2d})}+\|f\|
 _{L^1(\bbr^{2d})}\|f\|_{L^2(\bbr^{2d})}\Big),
 \end{aligned}
\end{equation*}
\begin{equation*}
 \begin{aligned}
 |H_3|\le&2\int_{\bbr^{2d}}\left|\int_{\bbr^{2d}}(\bw-\bv)\cdot \nabla_{\bx}j_{\e}(\bx-\bz, \by-\bw)\nabla_{\bx}f(t,\bz, \bw)d\bz d\bw\right||\langle \nabla_{\bx} f\rangle_{\e}|d\bx d\bv\\
 \le&C\|\nabla_{\bx} f\|_{L^2(\bbr^{2d})}^2,
 \end{aligned}
\end{equation*}
\begin{equation*}
 \begin{aligned}
 |H_4|\le&2\int_{\bbr^{2d}}\Bigg|\int_{\bbr^{2d}}\nabla_{\bv}\cdot L[f] \nabla_{\bx}f(t,\bz, \bw)j_{\e}(\bx-\bz, \by-\bw) d\bz d\bw \cdot \langle \nabla_{\bx} f\rangle_{\e}\Bigg|d\bx d\bv\\
 &+2\int_{\bbr^{2d}}\Bigg|\int_{\bbr^{2d}}\nabla_{\bv}j_{\e}(\bx-\bz, \by-\bw)\cdot \left[(L[f](\bz, \bw)-L[f](\bx, \bv))\otimes \nabla_{\bx}f(t,\bz, \bw)\right]\\
 &\qquad \qquad \quad d\bz d\bw\cdot \langle \nabla_{\bx} f\rangle_{\e}\Bigg|d\bx d\bv\\
 \le&C\|f\|_{L^1(\bbr^{2d})}\|\nabla_{\bx} f\|_{L^2(\bbr^{2d})}^2+ C\|(1+\bv^2)^{\frac12}f\|_{L^1(\bbr^{2d})}\|(1+\bv^2)^{\frac12}\nabla_{\bx}f\|_{L^2(\bbr^{2d})}
 \|\nabla_{\bx} f\|_{L^2(\bbr^{2d})},
 \end{aligned}
\end{equation*}
Substituting these estimates into \eqref{eq-vstesint} and integrating the resulting equation over $(t_1, t_2)$, $\forall t_1, t_2 \in [0,T]$, then letting $\e \to 0$, we obtain
\begin{equation} \label{eq-vsncont}
 \left| \|\nabla_{\bx}f(t_2)\|_{L^2(\bbr^{2d})}^2-\|\nabla_{\bx}f(t_1)\|_{L^2(\bbr^{2d})}^2\right|\le C(T)|t_2-t_1|,
\end{equation}
where we have used \eqref{eq-vsfles}. Combining \eqref{eq-vswctu} and \eqref{eq-vsncont}, we know
\[
 \nabla_{\bx}f \in C([0,T], L^2(\bbr^{2d}), \quad \forall T>0.
\]
Similarly we can prove $\nabla_{\bv}f \in C([0,T], L^2(\bbr^{2d}), \quad \forall T>0$. Due to the arbitrariness of $T$ and \eqref{eq-vsfles}, we can easily prove $f \in C([0,+\infty), X)$ by interpolation and
\[
 \|(1+\bv^2)^{\frac12}f(t)\|_{X} \le \|(1+\bv^2)^{\frac12}f_0\|_{X} \exp\left(\int_0^t c(\tau)d \tau\right), \quad a.e.\ t \ge 0.
\]
As for the solution sequences $\{f^{\sigma_j}\}$ and $\{g^{\sigma_j}\}$ with initial data $f_0$ and $g_0$, respectively, we have
\[
f^{\sigma_j} \rightharpoonup f,
\quad \text{
weakly-$\star$ in $L^{\infty}((0,T), X)\ \forall T>0$, as $\sigma_j \to 0$}
\]
and
\[
 g^{\sigma_j} \rightharpoonup g,
\quad \text{
weakly-$\star$ in $L^{\infty}((0,T), X)\ \forall T>0$, as $\sigma_j \to 0$},
\]
Thus we have
\[
 \sup_{0\le t \le T}\|f(t)-g(t)\|_{X}\le \|f_0-g_0\|_{X} \exp\left(\int_0^T c(t)dt\right), \quad \forall T\ge 0,
\]
which amounts to uniqueness of the strong solution. This completes the proof.

\begin{equation} \label{eq}
 \begin{aligned}
 \end{aligned}
\end{equation}
%
%
\section{Conclusion}
In this paper, we have developed a frame that can be used to established the well-posedness of weak, strong and classical solutions to the kinetic Cucker-Smale model with or without noise, no matter whether the initial data have compact support or not. Besides, we also rigorously justify the vanishing noise limit, which can be as a counterpart result to the vanishing viscosity method in hyperbolic conservation laws. Therefore we present complete theory for the kinetic Cucker-Smale model except the large-time behavior of the solution.

Our proof is based on weighted energy estimates and subtle compact analysis. The two weighted Hilbert spaces we introduced and the velocity averaging lemma in kinetic theory play important roles in our analysis. However, the time-asymptotic behavior of the solution is difficult to analyze. Maybe we can begin with some special situations. As for the kinetic Cucker-Smale model with noise, we guess the solution will tend to its steady state, if the initial perturbations are suitably small. This problem will be pursued in our future.
\bibliographystyle{plain}

\end{document}